\documentclass[a4paper,10pt]{article}
\usepackage{longtable,geometry}
\usepackage{graphicx}
\usepackage[english]{babel}
\usepackage[utf8]{inputenc}
\usepackage[T1]{fontenc}
\usepackage{amsmath,amsfonts,amsthm,amssymb}
\usepackage{bbm}
\usepackage{pict2e}
\usepackage[dvipsnames]{xcolor}

\usepackage[switch, modulo]{lineno}

\numberwithin{equation}{section} 

\newtheorem{theorem}{Theorem}[section]

\newtheorem{prop}[theorem]{Proposition}
\newtheorem{lemma}[theorem]{Lemma}

\newtheorem{remark}{Remark}[section]\theoremstyle{remark}
\newtheorem{definition}{Definition}[section]\theoremstyle{definition}

\newcommand{\Out}{\mathrm{out}}
\newcommand{\In}{\mathrm{in}}
\newcommand{\para}{{\scriptstyle \parallel}}
\newcommand{\gr}[1]{\boldsymbol{#1}}
\newcommand{\e}{\vec{\mathbf{e}}_1}
\newcommand{\ind}{\mathbbm{1}}
\renewcommand{\k}{{\vec{\kappa}}}

\newcommand{\tj}{{{\textit{\~{\j}}}}}
\newcommand{\ti}{{{\textit{\~{\i}}}}}

\title{Boltzmann-Grad limit of a hard sphere system in a box with diffusive boundary conditions}
\author{Corentin Le Bihan\footnote{UMPA (UMR CNRS 5669), \'Ecole Normale Superieur de Lyon, 46 allée d'Italie, 69364 LYON, FRANCE, e-mail: {\color{blue} {\tt corentin.le-bihan@ens-lyon.fr}}}}
\begin{document}
	\maketitle
	\begin{abstract}
		In this paper we present  a rigorous derivation of the Boltzmann equation in a compact domain with diffuse reflection boundary conditions. We consider a system of $N$ hard spheres of diameter $\epsilon$ in a box $\Lambda := [0,1]\times(\mathbb{R}/\mathbb{Z})^2$. When a particle meets the boundary of the domain, it is instantaneously reinjected into the box with a random direction, and conserving kinetic energy.
		We prove that the first marginal of the process converges in the scaling $N\epsilon^2  = 1$, $\epsilon\rightarrow 0$ to the solution of the Boltzmann equation, with the same short time restriction of Lanford's classical theorem.	
	\end{abstract}
	\section{Introduction}
	A simple model of gas is the hard sphere gas. Each molecule is described as a little sphere of diameter $\epsilon>0$ moving freely along straight lines in a domain $\Lambda\subset \mathbb{R}^d$, $d \geq 2$, and interacting with the other molecules only at distance $\epsilon$. A classical problem is the study of limits of such a system when the number of particles $N$ goes to infinity. Of course this limit depends deeply on the relation between $N$ and $\epsilon$. We shall focus here on an intermediate limit which bridges the microscopic scale proportional to $\epsilon$, and the macroscopic scales where we only see average quantities as the temperature or the mean velocity. 

	At a mesoscopic scale we look at the distribution $f(t,x,v)$ of one particle in the phase space $\Lambda_x\times\mathbb{R}^d_v$ at time $t \geq 0$. In the case of hard spheres, the only interesting limit of this type is the Boltzmann-Grad scaling $N\epsilon^{d-1}= 1$ (see \cite{Grad}), for which the density of the gas goes to $0$ as $\epsilon$, but the mean distance traveled by a particle between two collisions (the mean free path) is constant. In this regime we expect that the coordinates of two randomly chosen particles are "almost independent" and that the limiting one-particle distribution is governed by the Boltzmann equation
	\[\partial_t f+v\cdot\nabla_x f=Q(f,f) ,\]
	\[Q(f,f)(v) := \int_{\mathbb{R}^3\times\mathbb{S}^2}\left(f(v-\nu\cdot(v-v_*)\nu)f(v_*+\nu\cdot(v-v_*)\nu)-f(v)f(v_*)\right) b(v-v_*,\nu)dv_*d\nu\]
	where the operator $Q(f,f)(x,v)$ describes the variation due to collisions and $b(v-v_ *,\nu)$ is a given collision kernel (\cite{Boltzmann}).

	In a fundamental paper \cite{Lanford}, Lanford stated a convergence theorem of the hard spheres system for a short time $t<T^*$, $T^*$ depending on the initial condition. A detailed proof has been provided later on, see in particular \cite{GST,PSS} for quantitative bounds on the convergence error (including smooth potentials with finite range) in the case of a domain without boundary (see also \cite{Spohn,CIR,Delinger,GG}). There also exists a proof of long time convergence when the domain is $\mathbb{R}^2$ or $\mathbb{R}^3$ and the gas is very diluted (meaning in particular that the initial distribution of particles $f_0(x,v)$ is bounded by a Gaussian $\eta e^{-\frac{|x|^2+|v|^2}{2}}$, with $\eta >0$ sufficiently small) (see \cite{IP}).
	
	Adding boundaries is important especially in relation with the problem of nontrivial stationary solutions, which is one of the main domains of application of the Boltzmann equation (see \cite{C, EM}). Even without considering the stationary problem, the presence of boundaries leads to several delicate issues. A first one is the modeling itself; a problem which goes back to the origins of kinetic theory (see \cite{C} for a discussion on several different kinds of reflection law). 
	
	Let $\Lambda$ be an open domain of dimension $d$ with a smooth boundary. A first example of reflection law is the specular reflection: when a particle hits the boundary at point $x$ with velocity $v$, it is reflected with velocity $v' := v-2(n(x)\cdot v)n(x)$ where $n(x)$ is the inner normal vector of $\partial\Lambda$ in $x$. Thus at the boundary the distribution verifies the condition $f(t,x,v)= f(t,x,v')$. This dynamics encodes complications because of possible focusing. Note however Th\'{e}ophile Dolmaire's thesis (see  \cite{Dolmaire,Dolmaire2}) where a Lanford theorem in the half plain has been proved. 
	
	A second and very famous model are the Gaussian boundary conditions: when a particle meets the boundary at point $x$, it is reflected with velocity $v'$ following the probability law
	\[d\mathbb{P}(v') = (v'\cdot n(x))_+ {M_w(x,v')}dv',~~M_w(x,v):= \frac{e^{-\frac{|v|^2}{2T(x)}}}{(2\pi)^{\frac{d-1}{2}}T(x)^{\frac{d+1}{2}}}\] 
	where $T(x)$ represents the temperature at point $x$. Thus the distribution verifies for all $(x,v)$ such that $x\in\partial\Lambda$, $v\cdot n(x)>0$
	\[f(t,x,v) = M_w(x,v)\int  f(t,x,u)(u\cdot n(x))_-du.\]
	The system of hard spheres with these boundary conditions has fluctuating energy, at variance with the Lanford's setting which typically models an isolated system.
	
	The usual argument for the rigorous convergence in the Boltzmann-Grad limit looks hardly adaptable in this case. Note in particular, that we should preliminarily answer the following question (which can be found in \cite{C}): does the Boltzmann equation admit stationary solutions with prescribed temperature at the boundary? The question is not answered in full generality. However if the temperature at the boundary is smooth with small variations, then near to the hydrodynamic regime (the mean free path going to $0$) there exists a unique stationary solution (see \cite{EGKM,EM}). Unfortunately the scheme of Lanford's proof requires a priori estimates involving infinitely many reference distributions with increasing temperature (see for example chapter 5 of \cite{GST} or Section \ref{Well-posedness of the operators} of the present paper). In addition note that we do not even know if the hard sphere dynamics with Gaussian reflection is a well defined process (see \cite{Catapano}; however this looks just a technical problem and the process is indeed well defined for suitable smooth interactions (see \cite{GLP})).
	
	In the following we will therefore investigate a simpler model: the diffuse reflection in angle. In this model, a particle conserves its velocity when it reaches the boundary and is reflected in a random direction. More precisely, an incoming particle with coordinates $(x,v)$ has outcoming velocity $v'$ following the law
	\[d\mathbb{P}(v')=c_d \frac{(v'\cdot n(x))_+}{|v'|^d}\delta_{|v|-|v'|}dv',~~\mathrm{where}~c_d:=\left(\int_{\mathbb{S}^{d-1}}(\omega\cdot\e)_+d\omega\right)^{-1}\]
	with $\delta$ the Dirac mass, $c_d$ a normalization constant and $\e$ a unit vector. With this reflection law the distribution has to respect the following boundary condition: for $x\in\partial\Lambda$, $v\cdot n(x)>0$,
	\[f(t,x,v) = \int_{\mathbb{S}^{d-1}}f(t,x,|v|\omega)c_d(n(x)\cdot\omega)_-d\omega.\]
	Since the energy is conserved and any Gaussian distribution is a stationary measure, we can expect to be able to adapt Lanford's strategy. 
	
	This model can be seen as a model of rough boundary. 	
	In the hydrodynamic limit of the Boltzmann equation, it would lead to an adiabatic model: the temperature verifies Neumann's boudary condition $n(x) \cdot \nabla T_{\partial\Omega}=0$ and the mean velocities verify the Dirichlet boundary condition $u_{\partial\Omega}=0$ (as discussed in \cite{CLB}).
	
	The paper is organized as follows.
		
	We give a proper definition of the process in Section \ref{The Model} and we derive the evolution law of a symmetric distribution of particles in Section \ref{The BBGKY hierarchy and its pseudo-trajectories}: the BBGKY Duhamel series \eqref{BBGKY Duhamel serie}. Section \ref{Boltzmann's hierarchy and a priori estimates} is dedicated to the formal limit, namely the Boltzmann Duhamel series. In Section \ref{Main theorem and strategy of the proof} we state the Lanford's theorem in a domain with stochastic boundary (Theorem \ref{theorem de Landford}). Section \ref{Term by term convergence} is devoted to the main step of the proof:  the "mean" convergence of the hard sphere process to its formal limit (the \emph{punctual process}). 
	
	In the latter section we will have to restrict to the simple domain $\Lambda := [0,1]\times(\mathbb{R}/\mathbb{Z})^2$, to be able to study the geometry of the hard sphere process. One of the main ingredients is that outside a small set of particle configurations and of time variables, hard sphere and punctual process have the same velocities. Then the error between hard spheres and punctual particles comes only from shifts of size $\epsilon$ at each collision. If we look at a general  domain $\Lambda\subset \mathbb{R}^3$ we loose such simple feature. We believe that the theorem remains true, but the proof would be certainly more delicate.	
	
	\section{The model}\label{The Model}
	We will now give a precise definition of the process.
	
	We will use the notation $\gr{a}_n=(a_1,\cdots,a_n)$.
	
	We want to describe the motion of $N$ hard spheres of diameter $\epsilon$ in a smooth domain $\Lambda$. In the following we take $\Lambda=[0,1]\times\mathbb{T}^2$ where $\mathbb{T}:=\mathbb{R}/\mathbb{Z}$. The particles move along straight lines until they meet either the boundary of $\Lambda$ or another particle. When two particles meet, there is an elastic \emph{collision}. When a particle meets the boundary of $\Lambda$ at the point $x$ with incoming velocity $v^\In$, it is \emph{reflected} at the same point $x$ with velocity $v^\Out$ following a probability law $K_x(v^\In|v^\Out)(v^\Out\cdot n(x))_+dv^\Out$, where $n(x)$ is the inner normal vector of the surface $\partial\Lambda$. We say that there is a \emph{reflection}.
	
	It is not obvious that such process is well defined, and we will restrict in the following to the case of diffusion in angle : 
	\begin{multline}
	K_x(v^\In| v^\Out)(v^\Out\cdot n(x))_+dv^\Out = \frac{c_3 \delta(|v^\In|-|v^\Out|)(v^\Out\cdot n(x))_+dv^\Out}{|v^\Out|^3},\\
	\mathrm{with}~c_3 := \left(\int_{\mathbb{S}^2}(\omega\cdot \e)_+d\omega\right)^{-1} = \frac{1}{\pi}
	\end{multline}
	where $\e$ is the vector $(1,0,0)\in \mathbb{R}^3$.

	Note that in the case of diffusion in angle both measures $K_x(v^\In| v^\Out)(v^\Out\cdot n(x))_+dv^\Out$  and  $K_x(v^\In| v^\Out)(v^\In\cdot n(x))_-dv^\In$ are a probability measure. Thus we can in a certain sense inverse the hard sphere process. This will be the main ingredient of the construction of the process below. In the case of Gaussian boundary conditions, $K_x(v^\In |v^\Out)(v^\In \cdot n(x))_- dv^\In = M_w(x,v^\Out)(v^\In \cdot n(x))_-dv^\In$ which is not a probability measure and  the following strategy cannot be applied.
	\subsection{Construction of the stochastic process}\label{Construction of the stochastic process}
	To discuss the well-posedness of the system, we introduce the phase space 
	\begin{equation}
	\mathcal{D}_\epsilon^N :=\{(\gr{x}_N,\gr{v}_N)=(x_1,\cdots,x_N,v_1,\cdots v_N)\in \Lambda^N\times\mathbb{R}^{3N}~|~\forall i\neq j,~|x_i-x_j|>\epsilon\}
	\end{equation}
	and the probability space
	\begin{equation}
	\Omega := \{\bar{\omega}=(\omega^j)_{j\in \mathbb{Z}^*}~\mathrm{with}~\omega^j\in\mathbb{S}^2,~\omega^j\cdot\e>0 \mathrm{~if~} j>0,~\omega^j\cdot\e<0 \mathrm{~if~} j<0\}.
	\end{equation}
	We assign to $\Omega$ the measure $d\mathbb{P}(\bar{\omega}  = (\omega^j))$ which is the probability measure of  sequences of independent random variables such that $\omega^j$ follows the law $c_d(\omega^j\cdot \e)_-d\omega^j$ for $j <0$ and  $c_d(\omega^j\cdot \e)_+d\omega^j$ for $j >0$.
	
	Let $\Gamma$ be the function on $\partial\Lambda$ such that
	\[\Gamma(x) := \left\lbrace\begin{matrix} 1 \mathrm{~for~} x \in \{0\}\times\mathbb{T}^2\\  -1 \mathrm{~for~} x \in \{1\}\times\mathbb{T}^2. \end{matrix}\right.\]
	We define now a dynamics on the extended phase space $\mathcal{D}_\epsilon^N\times\Omega^N$. Let \[(\gr{z}^0_N,\gr{\bar{\omega}}^0_N)= (\gr{x}^0_N,\gr{v}^0_N,\bar{\omega}^0_1,\cdots,\bar{\omega}^0_N)\in \mathcal{D}^N_\epsilon\times\Omega^N\] and  $t>0$ be a time,
	\begin{itemize}
		\item $(\gr{z}^\epsilon_N(0),\bar{\gr{\omega}}^\epsilon_N(0))= (\gr{z}^0_N,\gr{\bar{\omega}}^0_N)$;
		\item until $\gr{z}_N^\epsilon(t)$ reaches the boundary of $\mathcal{D}_\epsilon^N$, $\gr{\bar{\omega}}^\epsilon_N(t)$ is constant and each ${z}^\epsilon_i(t)$ ($i\in\{1,\cdots,N\}$) moves along  straight lines;
		\item if $|x_i^\epsilon(t)-x_j^\epsilon(t)|=\epsilon$, $v_i^\epsilon(t^+)$ and $v_j^\epsilon(t^+)$ are given by an elastic collision between the two particles :
		\begin{equation}\label{ loi de collision}
		\left\lbrace\begin{split}
		v_i^\epsilon(t^+) = v_i^\epsilon(t^-) - \nu_{i,j}\cdot\left( v_i^\epsilon(t^-)-v_j^\epsilon(t^-)\right) \nu_{i,j}\\
		v_j^\epsilon(t^+) = v_j(t^-) + \nu_{i,j}\cdot\left( v_i^\epsilon(t^-)-v_j^\epsilon(t^-)\right) \nu_{i,j}
		\end{split}\right.
		\end{equation}
		with $\nu_{i,j}:=\frac{x_i^\epsilon(t)-x_j^\epsilon(t)}{\left|x_i^\epsilon(t)-x_j^\epsilon(t)\right|}$ and where $t^\pm$ indicate the limit from the future/past.
		\item if $x_i^\epsilon(t)$ meets the boundary of $\Lambda$, we record the in-coming direction of $v_i^\epsilon(t^-)$ and $v_i^\epsilon(t^+)$ takes the out-coming direction :
		\begin{equation}
		\left\lbrace~~~\begin{split}
		&v_i^\epsilon(t^+)=|v_i^\epsilon(t^-)|\Gamma(x_i^\epsilon(t))\omega^{\epsilon,1}_i(t^-)\\
		&\omega_i^{\epsilon,-1} = \Gamma(x^\epsilon_i(t)) \frac{v^\epsilon_i(t^-)}{|v^\epsilon_i(t^-)|}\\
		&\forall j\in\mathbb{Z}^*\setminus\{-1\},~\omega^{\epsilon,j}_i(t^+)=\omega^{\epsilon,j+1}_i(t^-);
		\end{split}\right.
		\end{equation}
	\end{itemize}
	then we iterate the process. For example in the case with only one particle and after $k$ reflections,
	\[\bar{\omega}^\epsilon_1(t) = \small\left(\begin{matrix}
	\cdots,&{\omega^{0,-2}_1},&\omega^{0,-1}_1,&\Gamma(x_1^1)\frac{v_1^1}{|v_1^1|},&\Gamma(x_1^2)\frac{v_1^2}{|v_1^2|},&\cdots,&\Gamma(x_1^k)\frac{v_1^k}{|v_1^k|}&,\omega_1^{0,k+1}&,\omega_1^{0,k+2}&,\cdots\\
	& {\scriptstyle{-k-2}}& \scriptstyle{-k-1}& \scriptstyle{-k}& \scriptstyle{-k+1}&& \scriptstyle{-1}& \scriptstyle{1}& \scriptstyle{2}
	\end{matrix}\right)\]
	where $x_1^i := x^\epsilon_1(t_i)$ and $v_1^i := v_1^\epsilon(t_i^-)$ are the position of the particle and its incoming velocity at the $i$-th reflection. 
	
	In the same way we can define the backward dynamics for $t<0$.
	
	Note that the variables $({\omega}_i^j)_{j<0}$ are used to record the reflections. This will be practical to reconstruct the dynamics backwardly, independently of the number of reflections.
	
	\subsection{Well-posedness of the process}
	We denote $\phi^{\epsilon,t}_N(\gr{z}_N,\gr{\bar{\omega}}_N)=(\gr{z}^\epsilon_N(t),\gr{\bar{\omega}}^\epsilon_N(t))$ the flow described above with initial conditions $(\gr{z}_N,\gr{\bar{\omega}}_N)$. It is not well defined everywhere and we can have \emph{bad initial data} which lead to
	\begin{itemize}
		\item a collision involving more than two particles at some time,
		\item two collisions/reflections at the same time,
		\item infinitely many collisions/reflections during a finite time,
		\item grazing collisions/reflections.
	\end{itemize}
	However such "pathological" trajectories are exceptional: if we denote $\mathcal{B}^N\subset \mathcal{D}_\epsilon^N\times\Omega^N$ the set of bad initial data and $\mathbb{P}^N(\gr{\bar{\omega}}_N):=\mathbb{P}(\bar{\omega}_1)\otimes\cdots\otimes\mathbb{P}(\bar{\omega}_N)$, we have
	\begin{prop}\label{definition de la dynamique}
		$\mathcal{B}^N$ is of zero measure for $d\gr{z}_Nd\mathbb{P}^N(\gr{\bar{\omega}}_N)$, the dynamics on $\mathcal{D}^N_\epsilon\times\Omega^N$ is well defined for almost all initial data and $\phi^{\epsilon,t}_N$ conserves the measure: for all Borel sets $A\subset\mathcal{D}^N_\epsilon\times\Omega^N$, $\phi^{\epsilon,t}_N(A\setminus\mathcal{B}^N)$ is measurable, with the same measure than $A$.
	\end{prop}\label{prop d'existence de la dynamique}
	This is an adaptation of the proof of Alexander \cite{Alexander} and it stems from the following lemma. Let $B^N_R$ be the euclidean ball of radius $R$ in $\mathbb{R}^{3N}$.
	\begin{lemma}\label{Lemme d'estimation des mauvais ensembles}
		Let $R>0$ be given, and let $\delta$ be a real number in $(0,\epsilon/2)$. Let 
		\begin{multline*}
		I:= \bigg\lbrace (\gr{z}_N,\bar{\gr{\omega}}_N)\in \Lambda^N\times B^N_R \times \Omega ^N \big\vert\\
		\mathrm{there~are~two~shocks~(reflections~or~collisions)~during~the~time~interval~}[0,\delta]\bigg\rbrace.
		\end{multline*}
		Then for  $\epsilon$ small enough, $|I|\leq C(N,\epsilon,R)\delta^2.$
	\end{lemma}
	\begin{proof}
		We treat the different cases separately.
		
		First consider the case with two shocks (reflections or collisions) in the same interval $[0,\delta]$, implying at least two different particles. 
		\[
		I_1:= \left\lbrace 
		(\gr{z}_N,\bar{\gr{\omega}}_N)\in \Lambda^N\times B^N_R \times \Omega ^N \left\vert
		\exists i\neq j,\mathrm{and~} k,l\notin \{i,j\},~ \left\lbrace\begin{split}d(x_i,B^1_\epsilon(x_k)\cup \partial\Lambda) \leq 2\delta R\\
		~d(x_j,B^1_\epsilon(x_l)\cup \partial\Lambda) \leq 2\delta R
		\end{split}\right.\right.\right\rbrace
		\]
		where $d(x,A)$ is the euclidean distance of a point $x$ to a subset $A$. $I_1$ is of measure at most $C(N,\epsilon,R)\delta^2$.
		
		Then we pass to the case where there is one particle which reflects twice the boundary on the interval $[0,\delta]$. This case is made impossible for $\delta<1/R$ because between two reflections a particle has to cross $(0,1)\times\mathbb{T}^2$.
	\end{proof}
	
	\begin{proof}[Proof of Proposition \ref{prop d'existence de la dynamique}]
		Fix $R>0$ and $t>0$. Let $\delta<\epsilon/2$ be a small parameter such that $t/\delta$ is an integer. Lemma \ref{Lemme d'estimation des mauvais ensembles} shows that there exists a subset $I_0(\delta,R)$ of $\Lambda^N\times B^N_R\times \Omega^N$ such that outside of $I_0$ there is at most one shock in the time interval $[0,\delta]$. Its measure is at most $C(N,\epsilon,R)\delta^2$. Observe that the set leading to grazing shocks is of zero measure.
		
		Note that the flow is conservative where it is well defined. Indeed the boundary conditions for collisions are conservative. For the reflection, the map \[(x,v,\omega^1)\mapsto\Big( x + \tau(x,v)v +(t-\tau(x,v))|v|\omega^1,|v|\omega^1,\omega^{-1}(t):=\Gamma(x + \tau(x,v)v)v/|v|\Big)\]
		where $\tau(x,v)$ is the time of travel of one particle to the boundary, is conservative. For the other $\omega^i, i\notin\{0,-1\}$, we just apply the shift $(\omega^i)\mapsto (\omega^{i+1})$, which is also conservative.
		
		Hence there is no pathological trajectory in $\Lambda^N\times B^N_R\times \Omega^N\setminus I_0(\delta,R)$. We recall that the set $\Lambda^N\times B^N_R\times \Omega^N$ is stable under the flow (the energy $\frac{1}{2}\sum_i |v_i|^2$ is conserved). We iterate the procedure and construct a set $I_1(\delta,R)$ such that outside $I_1$ there is at most one shock in the interval $[\delta,2\delta]$. Because the flow is conservative, $I_1(\delta,R)$ is of size at most $C(N,\epsilon,R)\delta^2$. More generally we can construct a sequence of sets $(I_k(\delta,R))_k$ such that outside $\bigcup_{0\leq k\leq K} I_k(\delta,R)$ there is no pathological trajectory during the interval $[0,(K+1)\delta]$.
		
		We define $I(\delta,t,R)$ as:
		\[I(\delta,t,R) := \bigcup_{0\leq k\leq t/\delta} I_k(\delta,R).\]
		$I(\delta,t,R)$ is of size at most $C(N,\epsilon,R)\delta^2\cdot t/\delta = C(N,\epsilon,R)t\delta$. Setting 
		\[I(t,R) := \bigcap_{n\in \mathbb{N}^*} I(t/n,t,R),\]
		$I(t,R)$ is of null-measure and outside it there is no pathological trajectory on $[0,t]$. We take the union $\mathcal{B}^N$ of the $I(t,R)$ for a sequence of $t$ and $R$ going to infinity. Outside it there is no pathological trajectory. This concludes the proof. 
	\end{proof}
	\section{The BBGKY hierarchy and its pseudo-trajectories}\label{The BBGKY hierarchy and its pseudo-trajectories}
	\subsection{Stochastic semigroup and expression of the hierarchy}
	We will use the following notation: for $1\leq k<l\leq n$ two integers, $\gr{a}_{k,l}=(a_k,\cdots,a_l)$.
	
	We want to study a system of $N$ identical particles, distributed at time zero according to a probability $\mu_N^0$ on $\mathcal{D}_\epsilon^N$. Because all the particles are indiscernible, the measure $\mu_N$ is assumed stable under permutation of particle labels.
	
	We define by duality the semi-group $T_N^\epsilon(t)$ on $\mathcal{M}_0(\mathcal{D}_\epsilon^N)$, the space of finite signed measures $\mu_N$ such that the set of bad trajectories $\mathcal{B}^N\subset\mathcal{D}^N_\epsilon\times\Omega^N$ is of measure zero for $\mu_N\otimes \mathbb{P}^N$. For any bounded continuous function $\varphi : \mathcal{D}_\epsilon^N\rightarrow\mathbb{R}$, 
	\begin{equation}
	\int_{\mathcal{D}_\epsilon^N} \varphi(\gr{z}_N) d(T_N^\epsilon(t)\mu_N)(\gr{z}_N) := \int_{\mathcal{D}_\epsilon^N} \left(\int_{\Omega^N}  \varphi\left[\phi^{\epsilon,t}_N(\gr{z}_N^0,\gr{\bar{\omega}}^0_N)\right]d\mathbb{P}^N(\gr{\bar{\omega}}^0_N)\right) d\mu_N(\gr{z}_N^0).
	\end{equation}
	In the case of density measures, we have an explicit formula :
	\begin{prop}
		For $W_{0,N}\in L^1\cap L^\infty(\mathcal{D}_\epsilon^N)$,
		\begin{equation}\label{caractéristique proba}
		T_N^\epsilon(t)(W_{0,N}(\gr{z}_N)d\gr{z}_N) = \left( \int_{\Omega^N}W_{0,N}\left[\phi^{\epsilon,-t}_N(\gr{z}_N^0,\gr{\bar{\omega}}^0_N)\right]d\mathbb{P}^N(\gr{\bar{\omega}}^0_N)\right)d\gr{z}_N^0,\mathrm{~with~}t>0.
		\end{equation}
	\end{prop}
	\begin{proof}
		It is a direct application of the conservation of measure of $\phi^{\epsilon,t}_N$, Proposition \ref{definition de la dynamique}.
	\end{proof}
	
	We denote $W_{0,N}$ the initial density distribution of particles at time $0$ and $W_N(t)$ the evolution of this distribution. To observe some limit behavior, we have to fix the number of particles $s$ that we study. Let $(f_N^{(s)})_{1\leq s\leq N}$ be the marginals of $W_{0,N}$ and $(f_{0,N}^{(s)}(t))_{1\leq s\leq N}$ be the marginals of $W_N(t)$:
	\begin{equation}
	\forall \gr{z}_s\in\mathcal{D}_\epsilon^s,~f_{0,N}^{(s)}(\gr{z}_s) := \int W_{0,N}(\gr{z}_s,\gr{z}_{s+1,N})d\gr{z}_{s+1,N}
	\end{equation}
	and the same thing for $f_N^{(s)}(t)$. By convention we extend functions on $\mathcal{D}_\epsilon^s$ by $0$ outside $ \mathcal{D}_\epsilon^s$.
	
	The following theorem describes the evolution of the marginals:
	\begin{theorem}\label{théorème sur validité de BBGKY}
		Let $W_{0,N}$ be a function in $L^\infty\cap L^1(\mathcal{D}^N_\epsilon)$. Its marginals $(f^{(s)}_{0,N})_s$ verify the series expansion of the BBGKY hierarchy (we will call it in the following the BBGKY hierarchy):
		\begin{multline}\label{BBGKY Duhamel serie}
		f_N^{(s)}(t) = \sum_{r =0}^{N-s}\alpha(N-s,r)\epsilon^{2r} \int_0^t\int_0^{t_1} \cdots\int_0^{t_{r-1}}dt_1\cdots
		dt_r T_s^{\epsilon}(t-t_1)C_{s,s+1}^\epsilon\cdots\\
		\cdots C_{s+r-1,s+r}^\epsilon T^\epsilon_{s+r}(t_r)f^{(s+r)}_{0,N}=:\sum_{r=0}^{N-s} \alpha(N-r,N) \epsilon^{2r}Q_{s,s+r}^\epsilon(t)f^{(s+r)}_{0,N}
		\end{multline}
		where $\alpha(r,s) = r(r-1)\cdots(r-s+1)$ and $C^\epsilon_{s,s+1}$ is the collision operator:
		\begin{equation}
		C_{s,s+1}^\epsilon := C_{s,s+1}^{\epsilon, +} -C_{s,s+1}^{\epsilon,-}
		\end{equation}
		\begin{equation}
		C_{s,s+1}^{\epsilon, +} f^{(s+1)}(\gr{z}_s) := \sum_{i=1}^s \int_{\mathbb{S}^{2}\times\mathbb{R}^3} f^{(s+1)}(\cdots,x_i,v'_i,\cdots,x_i+\epsilon \nu,v_{s+1}')(\nu\cdot(v_{s+1}-v_i))_+d\nu dv_{s+1}
		\end{equation}
		\begin{equation}
		C_{s,s+1}^{\epsilon, -} f^{(s+1)}(\gr{z}_s) := \sum_{i=1}^s \int_{\mathbb{S}^{2}\times\mathbb{R}^3} f^{(s+1)}(\gr{z}_s,x_i+\epsilon \nu,v_{s+1})(\nu\cdot(v_{s+1}-v_i))_-d\nu dv_{s+1}
		\end{equation}
		and $(v_i,v_{s+1})$ is the scattering of $(v_i',v_{s+1}')$ (see Equation \eqref{ loi de collision}).
	\end{theorem}
	
	The strategy of the proof (presented in section \ref{Proof of the theorem théorème sur validité de BBGKY} below) is an adaptation of \cite{SP}.
	
	\subsection{The pseudotrajectories development} 
	We begin by rewriting \eqref{BBGKY Duhamel serie} with an explicit "characteristic" formula associated to the \emph{interacting process}.
	
	We define for $r,s\in\mathbb{N}$ the \emph{collision tree} $a=(a(k))_{s< k \leq s+r}$. It is a finite sequence such that a $a(k)$ is in $\{1,\cdots k-1\}$, and we denote $\mathfrak{A}_s^{r+s}$ the set of all collision trees. We construct the \emph{pseudotrajectory} $\zeta^\epsilon( \tau, t,\gr{z}_s, \gr{(t,\nu,\bar{v})}_{s+1,s+r}, \gr{\bar{\omega}}_{r+s} ,a,\gr{\sigma})$ for $\tau\in[0,t]$, $\gr{(t,\nu,\bar{v})}_{s+1,s+r} := (t_i,\nu_i,\bar{v}_i)_{s<i\leq s+r}\in (\mathbb{R}\times\mathbb{S}^2\times\mathbb{R}^3)^r$  with $t>t_{s+1}\cdots>t_{s+r}>0$, $\gr{\bar{\omega}}_{r+s}=(\bar{\omega}_1,\cdots,\bar{\omega}_{r+s})\in\Omega^{r+s}$, $a\in\mathfrak{A}_s^{s+r}$ and $\gr{\sigma} := (\sigma_{s+1},\cdots,\sigma_{s+r})\in \{\pm1\}^r$. The number of particles of $\zeta^\epsilon(\tau)$ is not constant: for $\tau$ between $t_k$ and $t_{k+1}$ there are $s+k$ particles (by convention, $t_s:=t$ and $t_{r+s+1}=0$). Finally we denote $\gr{\bar{\omega}}_{r+s}^\epsilon(\tau)$ the evolution of the reflection parameters. We define $(\zeta^\epsilon(\tau),\gr{\bar{\omega}}^\epsilon_{r+s}(\tau):= (x^\epsilon_1(\tau),v^\epsilon_1(\tau),\cdots,x^\epsilon_{s+r(\tau)}(\tau),v^\epsilon_{s+r(\tau)}(\tau),\gr{\bar{\omega}}^\epsilon_{r+s}(\tau))$ (the number of particles depends on time) by
	\begin{itemize}
		\item $(\zeta^{\epsilon}(t),\gr{\bar{\omega}}_{r+s}^\epsilon(t)) := (\gr{z}_s,\gr{\bar{\omega}}_{r+s})$
		\item  for $\tau \in(t_{k+1},t_k)$, $\gr{\bar{\omega}}^\epsilon_{k+1,r+s}(\tau)$ is constant and
		\[(\zeta_{1,k}^\epsilon(\tau),\gr{\bar{\omega}}^\epsilon_{1,k}(\tau))=\phi^{\epsilon,(\tau-t_k)}_{k}\left(\zeta_{1,k}^\epsilon(t_k^+), \gr{\bar{\omega}}^\epsilon_{1,k}(t_k^+)\right)\]
		\item at time $t_k^+$, a particle is added at the position  $x^\epsilon_{k}(t_k)=x^\epsilon_{a(k)}(t_k)+\epsilon\nu_k$. If $\sigma_k =1$, we will have $\nu_k\cdot(\bar{v}_{k}-v^\epsilon_{a(k)}(t_k^+))>0$ and the velocities $(v_{a(k)}^\epsilon(t_k^-),{v}^\epsilon_{k}(t_k^-))$ are given by the usual scattering
		\[\left\lbrace\begin{split}
		&v^\epsilon_{a(k)}(t_k^-) = v^\epsilon_{a(k)}(t^+_k) - \nu_k\cdot\left(v^\epsilon_{a(k)}(t^+_k)-\bar{v}_{k}\right)\nu_k\\
		&v_{k}^\epsilon(t_k^-)= \bar{v}_{k}+\nu_k\cdot\left(v_{a(k)}^\epsilon(t^+_k)-\bar{v}_{k}\right)\nu_k
		\end{split}\right.\] 
		Else if $\sigma_k =-1$, $\nu_k\cdot(\bar{v}_{k}-v^\epsilon_{a(k)}(t_k^+))<0$ and we will have no scattering. The velocities at time $t_k^-$ are just $(v_{a_{k}}^\epsilon(t_k^+),\bar{v}_{k})$,
		\item Here and below, with a slight abuse of notation, $(x^\epsilon(\tau),v^\epsilon(\tau))$ designate the coordinates of pseudotrajectories (which are different from coordinates of the stochastic trajectories introduced in section \ref{Construction of the stochastic process}).
	\end{itemize}
	
	We denote $\mathcal{G}^\epsilon(\gr{z}_s,t,a,\gr{\sigma})$ the set of admissible coordinates, \textit{i.e.} the $(\gr{(t,\nu,\bar{v})}_{s+1,s+r},\gr{\bar{\omega}}_{r+s})$ such that the pseudo-trajectory is well defined according to the previous iteration, and 
	\[d\Lambda^\epsilon_{a,\gr{\sigma}}(\gr{(t,\nu,\bar{v})}_{s+1,s+r},\gr{\bar{\omega}}_{s+r}) = \left(\prod_{k=s+1}^{s+r} \left[\sigma_k \nu_k\cdot(\bar{v}_{k}-v^\epsilon_{a(k)}(t_k^+))\right]_+ dt_k d\bar{v}_{k}d\nu_k\right)d\mathbb{P}^{s+r}(\gr{\bar{\omega}}_{s+r}).\]
	Then the formula \eqref{BBGKY Duhamel serie} becomes
	\begin{multline}
	f_N^{(s)}(t,\gr{z}_s)= \sum_{r=0}^{N-s}\alpha(N-s,r)\epsilon^{2r}\sum_{a\in\mathfrak{A}_s^{s+r},\gr{\sigma}} \sigma_{s+1}\cdots\sigma_{s+r}\\
	\times\int_{\mathcal{G}^\epsilon(\gr{z}_s,t,a,\gr{\sigma})}d\Lambda^\epsilon_{a,\gr{\sigma}}(\gr{(t,\nu,\bar{v})}_{s+1,s+r},\gr{\bar{\omega}}_{s+r}) f_{0,N}^{(s+r)}(\zeta^\epsilon(0)).
	\end{multline}

	\subsection{Proof of Theorem \ref{théorème sur validité de BBGKY}}\label{Proof of the theorem théorème sur validité de BBGKY}
	Now we give a dual form of the previous equation. The idea is to look at the application $(\gr{z}_s,\gr{(t,\nu,\bar{v})}_{s+1,s+r},\gr{\bar{\omega}}_{s+r})\mapsto (\zeta^\epsilon(0),\gr{\bar{\omega}}^\epsilon(0))$ from $\mathcal{D}_\epsilon^s\times \mathcal{G}^\epsilon(a,\gr{\sigma})$ to $(\mathcal{D}^{s+r}_\epsilon\times\Omega^{s+r})\setminus\mathcal{B}^{s+r}
	$ which is a local homeomorphism that sends the measure $\epsilon^{2r} d\gr{z}_sd\Lambda^\epsilon$ into $d\zeta^\epsilon(0)d\mathbb{P}^{r+s}(\gr{\bar{\omega}}^\epsilon(0))$. It is not injective since an initial data can give different pseudotrajectories depending on whether a collision is seen as the \emph{creation} of a particle in the pseudotrajectory or as a \emph{recollision} (a collision between two particles that already exist). Nevertheless the number of collisions is locally constant and finite. Then indexing on $a$, $\gr{\sigma}$ and a new discrete parameter $M\in[1,\bar{M}]\subset\mathbb{N}$, we can define the inverse flow $(\zeta_s^{b,\epsilon},\gr{\bar{\omega}}_{s+r}'')(t,\gr{z}_{s+r},\gr{\bar{\omega}}_{s+r},a,\gr{\sigma},M)$. Because the number of collisions is almost surely finite, $\bar{M}$ is also almost surely finite and locally constant (we put $\bar{M}=0$ outside the image of the homeomorphism). Thus the equation \eqref{BBGKY Duhamel serie} can be rewritten in the weak sense as for all bounded continuous functions $\varphi$ defined on $\mathcal{D}^s_\epsilon$,
	\begin{multline}
	\int_{\mathcal{D}_\epsilon^s} \varphi(\gr{z}_s)f^{(s)}_N(t,\gr{z}_s)d\gr{z}_s=\sum_{r=0}^{N-s}\alpha(N-s,r)\sum_{a\in\mathfrak{A}_s^{s+r},\gr{\sigma}} \sigma_{s+1}\cdots\sigma_{s+r}\\
	\times\int_{\mathcal{D}_\epsilon^{s+r}\times\Omega^{r+s}}\sum_{M=1}^{\bar{M}}\varphi(\zeta^{b,\epsilon}_s(t))f^{(s+r)}_{0,N}(\gr{z}_{s+r}) d\gr{z}_{s+r}d\mathbb{P}^{s+r}(\gr{\bar{\omega}}_{s+r}).
	\end{multline}
	
	To check this equality, we prove it in the probability space. $\Omega$ is a compact metric space as countable product of compact spaces, so $\mathcal{D}_\epsilon^N\times\Omega^N$ is a Polish space.
	
	We find an analogue to $T_N^\epsilon$ on $\mathcal{D}_\epsilon^N$. Let $\mu_{0,N}$ be a measure in $\mathcal{M}_0\left(\mathcal{D}_\epsilon^N\times\Omega^N\right)$, the set of finite signed measures stable under permutation of variables and such that $\mathcal{B}^N$ has zero measure. We define $H_N(t)\mu_{0,N}$ by duality: for each $\varphi$ bounded continuous function on $\mathcal{D}^N_\epsilon\times\Omega^N$,
	\begin{equation}
	\int_{\mathcal{D}_\epsilon^N\times\Omega^N} \varphi(\gr{z}_N,\gr{\bar{\omega}}_N)d(H_N(t)\mu_{0,N})(\gr{z}_N,\gr{\bar{\omega}}_N):= \int_{\mathcal{D}_\epsilon^N\times\Omega^N}\varphi(\phi^{\epsilon,t}_N(\gr{z}_N,\gr{\bar{\omega}}_N)) d\mu_{0,N}(\gr{z}_N,\gr{\bar{\omega}}_N).
	\end{equation}
	
	Next, let $(\mu_{0,N}^{(s)})_{1\leq s\leq N}$ be the marginals of $\mu_{0,N}$, and $(\mu_N^{(s)}(t))_{1\leq s\leq N}$ the marginal of $H_N(t)\mu_{0,N}$. Then it suffices to prove that for all bounded continuous functions $\varphi_s$ on $\mathcal{D}_\epsilon^s\times\Omega^s$, we have
	\begin{equation}
	\begin{split}
	\int_{\mathcal{D}_\epsilon^s\times\Omega^s}&\varphi_s(\gr{z}_s,\bar{\gr{\omega}}_s) d\mu_N^{(s)}(t)(\gr{z}_s,\bar{\gr{\omega}}_s)\\
	&=\sum_{r=0}^{N-s}\alpha(N-s,r)\sum_{a\in\mathfrak{A}_s^{s+r},\gr{\sigma}} \sigma_{s+1}\cdots\sigma_{s+r}\int_{\mathcal{D}_\epsilon^{s+r}\times\Omega^{s+r}}\sum_{M=1}^{\bar{M}}\varphi_s(\zeta^{b,\epsilon}_s(t),\gr{\bar{\omega}}''_{s}(t)) d\mu_{0,N}^{(s+r)}\\
	&=: \sum_{r=0}^{N-s}\int_{\mathcal{D}_\epsilon^{r+s}\times\Omega^{r+s}}\varphi_s(\gr{z}_s,\gr{\bar{\omega}}_s)d\mathcal{T}_{s,s+r}(t)(\mu_{0,N}^{s+r})(\gr{z}_s,\gr{\bar{\omega}}_s)
	\end{split}
	\end{equation}
	where the $\mathcal{T}_{s,s+r}(t)(\mu_{0,N}^{(s+k)})$ are defined by duality.
	
	To prove this equality, it is sufficient to prove it for elementary measures.
	
	Let $(\gr{z}_N,\gr{\bar{\omega}}_N)\in\mathcal{D}^N_{\epsilon}\times\Omega^N$ be such that the dynamics is well defined. We consider the measure
	\[\Delta(\gr{z}_N,\gr{\bar{\omega}}_N)(\zeta_1,\bar{\omega}'_1,\cdots,\zeta_N,\bar{\omega}'_N) := \frac{1}{N!} \sum_{\sigma\in \mathfrak{S}_N^N}\prod_{i=1}^N \delta_{(\zeta_i,\bar{\omega}'_i)=(z_{\sigma(i)},\bar{\omega}_{\sigma(i)})}\]
	where $\mathfrak{S}_k^N$ is the set of injection of $\{1,\cdots,k\}$ into $\{1,\cdots,N\}$. Its $k$-th marginal is 
	\[\left(\Delta(\gr{z}_N,\gr{\bar{\omega}}_N)\right)^{(k)}(\zeta_1,\bar{\omega}'_1,\cdots,\zeta_k,\bar{\omega}'_k) = \frac{(N-k)!}{N!}\sum_{\sigma\in \mathfrak{S}_k^N}\prod_{i=1}^k \delta_{(\zeta_i,\bar{\omega}'_i)=(z_{\sigma(i)},\bar{\omega}_{\sigma(i)})}\]
	and we have immediately its evolution with respect to the hard spheres dynamics:
	\[H_N(t)\Delta(\gr{z}_N,\gr{\bar{\omega}}_N) = \Delta\left(\phi^{\epsilon,-t}_N(\gr{z}_N,\gr{\bar{\omega}}_N)\right). \]
	
	First there exists a finite sequence of times $0=t_0<t_1<\cdots<t_K = t$ such that on each segment $[t_i,t_{i+1}]$ there is only one collision or a reflection. In addition we impose that if the collision or the reflection implied the particles $k$ and $l$,
	\[\left\lbrace\begin{split}
	&\phi^{\epsilon,\tau}_{N-1}(\gr{z}^\epsilon_{1,k-1}(t_i),\gr{z}^\epsilon_{k+1,N}(t_i),\gr{\bar{\omega}}^\epsilon_{1,k-1}(t_i),\gr{\bar{\omega}}^\epsilon_{k+1,N}(t_i) )\\
	&\phi^{\epsilon,\tau}_{N-1}(\gr{z}^\epsilon_{1,l-1}(t_i),\gr{z}^\epsilon_{l+1,N},\gr{\bar{\omega}}^\epsilon_{1,l-1}(t_i),\gr{\bar{\omega}}^\epsilon_{l+1,N}(t_i) )
	\end{split}\right.\]
	moves like free flow for $\tau \in [0,t_{i+1}-t_i]$. It is possible to construct such sequence because the free flow is well defined and continuous.
	
	Then it is sufficient, in view of the \emph{semigroup property} verified by the marginals, to prove our assumption only for a segment $[0,t_1]$.
	
	To simplify the notation, in the following $\left(\Delta(\gr{z}_N,\gr{\bar{\omega}}_N)\right)^{(k)}$ will be denoted $\Delta_k$. We will prove the formula only for $\Delta_1$, the other cases work in the same way (see \cite{SP} for more details).
	
	Because there is at most one collision in $[0,t_1]$, all the $\mathcal{T}_{1,1+r}(t)\Delta_{1+r}$ vanish for $r\geq 2$.
	
	If there is no collision, then $\Delta_1(t)$ is just the push-forward of $\Delta_1$ by the free flow with diffusion, and $\mathcal{T}_{1,2}(t)\Delta_2$ vanishes. So the formula is verified.
	
	If there is one collision, we can assume without loss of generality that it occurs between particles $1$ and $2$. Then in $\mathcal{T}_{1,1}(t)\Delta_1$ the two first particles are replaced by virtual ones :
	\[\mathcal{T}_{1,1}(t)\Delta_1 = \frac{1}{N}\left(\delta_{(\tilde{z}_1(t),\bar{\omega}_1)} + \delta_{(\tilde{z}_2(t),\bar{\omega}_2)}+\sum_{i=3}^N \delta_{(z_i(t),\bar{\omega}_i(t))}\right),~~\mathrm{with~}\tilde{z}_i(t)=\left(x_i-tv_i,v_i\right). \]
	
	In $\mathcal{T}_{1,2}(t)\Delta_2$, there are two parts: a first part which corresponds to post collision directions : $\frac{1}{N}\left(\delta_{z_1(t),\bar{\omega}_1(t)} +\delta_{z_2(t),\bar{\omega}_2(t)}\right)$, and a second corresponding to the negative Dirac mass in the virtual particles : $-\frac{1}{N}\left(\delta_{(\tilde{z}_1(t),\bar{\omega}_1)} + \delta_{(\tilde{z}_2(t),\bar{\omega}_2)}\right)$, which compensate the previous error.
	
	Finally we get the expected formula.\qed
	
	\section{Boltzmann's hierarchy and \textit{a priori} estimates}\label{Boltzmann's hierarchy and a priori estimates}
	\subsection{Definition of the Boltzmann hierarchy}
	We want now to describe a formal limit of the BBGKY hierarchy \eqref{BBGKY Duhamel serie} when $\epsilon$ tends to $0$ in the scaling $\epsilon^{2}N =1$. For $(f^s_0)_s$ a family of symmetric functions on $(\Lambda\times\mathbb{R}^3)^s$, we define the Boltzmann hierarchy and its series expansion (what we call later the \emph{Boltzmann hierarchy}):
	\begin{equation}
	\begin{split}
	f^s(t) &= \sum_{r=0}^\infty \int_0^t\cdots\int_0^{t_{r-1}}dt_1\cdots dt_r T^0_{s}(t-t_1)C^0_{s,s+1}\cdots C^0_{s+r-1,s+r}T^0_{s+r}(t_r)f_0^{s+r}\\
	&=:\sum_{r=0}^\infty Q_{s,s+r}^0(t)f^{s+r}_0
	\end{split}
	\end{equation}
	where $T^0_{s}(t)$ is the semigroup associated to the dynamics of $s$ punctual particles with reflection in angle (and no collision) and $C^0_{s,s+1}$ is the formal collision operator for punctual spheres :
	\begin{equation}
	C_{s,s+1}^0 := C_{s,s+1}^{0, +} -C_{s,s+1}^{0,-}
	\end{equation}
	\begin{equation}
	C_{s,s+1}^{0, +} f^{s+1}(\gr{z}_s) := \sum_{i=1}^s \int_{\mathbb{S}^{2}\times\mathbb{R}^3} f^{s+1}(\cdots,x_i,v'_i,\cdots,x_i,v_{s+1}')(\nu\cdot(v_{s+1}-v_i))_+d\nu dv_{s+1}
	\end{equation}
	\begin{equation}
	C_{s,s+1}^{0, -} f^{s+1}(\gr{z}_s) := \sum_{i=1}^s \int_{\mathbb{S}^{2}\times\mathbb{R}^3} f^{s+1}(\gr{z}_s,x_i,v_{s+1})(\nu\cdot(v_{s+1}-v_i))_-d\nu dv_{s+1}.
	\end{equation}
	
	Remark that for initial data of the form $(f^{\otimes s}_0)$, the Boltzmann hierarchy is of the form $(f^{\otimes s}(t))$, where $f(t)$ is solution of :
	\begin{equation}
	f(t) = T_1^0(t) f_0 + \int_0^t T_1^0(\tau) C_{1,2}^0 f(\tau)^{\otimes 2} d\tau
	\end{equation}
	which is precisely the Boltzmann equation in the integral form.
	
	To properly define this operator, we have to find a nice functional space on which the $(f_0^s)$ will be defined :
	\begin{definition}
		For $\beta>0$ and $\mu$ two constants, we define the Banach space $X_{\beta,\mu}$ such that $(f^s)\in X_{\beta,\mu}$ if and only if for all $s \in \mathbb{N}^*$, $f^s$ is measurable, symmetric, compatible:
		\begin{equation}
			\forall s\in \mathbb{N}^*,~\forall\gr{z}_s\in\left(\Lambda\times\mathbb{R}^3\right)^{s},~f^s(\gr{z}_s) = \int_{\Lambda\times\mathbb{R}^3} f^{s+1}(\gr{z}_s,z_{s+1})dz_{s+1}
		\end{equation} and 
		\begin{equation}
		\left\Vert (f^s)_s\right\Vert_{\beta,\mu} := \sup_{s\in\mathbb{N}^*} \underset{\gr{z}_s\in\left(\Lambda\times\mathbb{R}^3\right)^s}{\mathrm{essup}} |f^s(\gr{z}_s)|\exp\left(\mu s +\frac{\beta}{2}\|\gr{v}_s\|^2\right)
		\end{equation}
		is finite, with $\|\gr{v}_s\|^2 = \sum_{i=1}^s|v_i|^2$.
		
		We denote the closed subspace of continuous functions $\tilde{X}_{\beta,\mu}:= X_{\beta,\mu}\cap\prod_{s\geq 1} \mathcal{C}((\Lambda\times\mathbb{R}^3)^s)$.
	\end{definition}	

	\begin{theorem}\label{existence de la hierarchie de Boltzmann}
		For $\beta>0,~\mu$, there exist $\beta'>0,~ \mu'$ and a time $T$ such that the Boltzmann hierarchy 
		\[\left\lbrace\begin{split}
		&\tilde{X}_{\beta,\mu} \rightarrow C\left([0,T],X_{\beta',\mu'}\right)\\
		&(f_0^s)_s\mapsto (f^s(t))_s
		\end{split}\right.\]
		is continuous.
	\end{theorem}
	The rest of this section is devoted to the proof of the theorem. 
	
	In a first time we will prove that $Q^0_{s,s+r}(t)f_0^{s+r}$ is well defined if $(f_0^s)$ is in $\tilde{X}_{\beta,\mu}$. Then the continuity estimates of \cite{GST} and \cite{Lanford} show that the sum is well defined and continuous.
	
	\subsection{Well-posedness of the operators $Q_{s,s+r}^0(t)$}\label{Well-posedness of the operators}
	The main difficulty is that the transport semigroup $T_1^0(t)$ does not send continuous functions onto continuous functions: discontinuities will appear in the future of $\{0\}_t\times\partial\Lambda\times\mathbb{R}^3$. Thus we have to check that we can apply $C_{s,s+1}^0$ on the function $T_{s+1}^0(t)f^{s+1}_0$.
	
	We will restrict ourselves to the case where $f_0^{s}= g_1\otimes\cdots\otimes g_s$. We define the \emph{future set} $\mathcal{F}(t)\subset \mathbb{R}^+\times\Lambda\times\mathbb{R}^3$ as 
	\begin{equation}
	\mathcal{F}(t) := \left\lbrace (\tau,x,v)|x-(\tau-t)v\in\partial\Lambda,~\tau\geq t\right\rbrace
	\end{equation}
	and $\mathcal{F}(t)\cap\{t =0\} := \{(x,v)\in\mathbb{R}^3,~(0,x,v)\in\mathcal{F}(t)\}$
	
	Then we have the following lemma, equivalent to a "weak" propagation of continuity:
	\begin{lemma}
		Let $g$ be a measurable function on $\Lambda\times\mathbb{R}^3$, continuous outside $(\mathcal{F}(t_1)\cup\cdots\cup\mathcal{F}(t_r))\cap\{t=0\}$ for $(t_1,\cdots,t_r)\in (\mathbb{R}_-^*)^s$ and dominated by $Ce^{-\beta|v|^2/2}$ for some $\beta>0$. Then $T^0_1 g :(t,x,v)\mapsto (T^0_1(t)g)(x,v)$ is continuous outside $\mathcal{F}(0)\cup\mathcal{F}(t_1)\cup\cdots\cup\mathcal{F}(t_r)$ and dominated by $Ce^{-\beta|v|^2/2}$.
	\end{lemma}
	\begin{proof}
		We will show that the trace of $T^0_1 g$ on  $\Sigma^+ := \{(x,v)\in\partial\Lambda\times\mathbb{R}^3 ,~ v\cdot n(x)>0\}$ is continuous. In the following we will denote $t_0:=0$.
		
		We define the function $g_b$ on $\mathbb{R}^+\times \Sigma^+$ by :
		\[g_b(t,x,v) := \int_{\mathbb{S}^2}g(x-t|v|\omega,|v|\omega)(\omega\cdot n(x))_-c_d d\omega \]
		with $g$ extended by $0$ outside $\Lambda$. A particle reflected with velocity $|v|\omega$ crosses the domain in time $1/|\omega\cdot\e||v|$ and $g$ is continuous outside $(\mathcal{F}(0)\cup\mathcal{F}(t_1)\cup\cdots\cup\mathcal{F}(t_r))\cap\{t=0\}$. Thus the function $\omega\mapsto g(x-t|v|\omega,|v|\omega)$ is continuous outside the set of directions $\omega$ such that the particle meets the boundary at time $t_i$ for some $i$: 
		\[\bigcup_{i=0}^r\left\lbrace\omega\in\mathbb{S}^2,\frac{1}{|\omega\cdot \e||v|}= (t-t_i)\right\rbrace\]
		which is of zero measure in $\mathbb{S}^2$. Thus $g_b$ is continuous on $\mathbb{R}^+\times\Sigma^+$.
		Now we define the operator $B :C(\mathbb{R}^+\times \Sigma^+)\rightarrow C(\mathbb{R}^+\times \Sigma^+)$ by
		\[Bf(t,x,v):= \int_{\mathbb{S}^{2}} f\left(t-\frac{1}{|\omega\cdot\e||v|},x-\frac{\omega}{|\omega\cdot\e|},|v|\omega\right) (n(x)\cdot \omega)_-c_d d\omega\]
		where  $f$ is extended by $0$ for $t<0$. Using the same argument than for $g_b$, $Bf$ is continuous. $Bf$ is the image of a distribution $f$ on $\mathbb{R}^+\times\Sigma^+$ after one more reflection.
		
		Applying the formula \eqref{caractéristique proba} in the case of one particle, we obtain that the trace of $T_1^0g$ on $\mathbb{R}^+\times\Sigma^+$ is 
		\begin{equation}\label{Série pour la trace} g_t :=\sum_{i=0}^\infty B^ig_b.\end{equation}
		
		We have to show that this sum converges normally on every compact set. In fact the sum is locally finite: 
		\begin{lemma}\label{estimation des reflections}
			For $(x,v)\in\Lambda\times \mathbb{R}^3$, the stochastic process will have at most $t|v|$ reflections.
		\end{lemma}
		\begin{proof}
			To have $N$ reflections, a particle has to travel at least a distance $N$, so $N\leq t|v|$.
		\end{proof}
		
		From this bound we deduce that the series \eqref{Série pour la trace} is finite on every compact set, and so $g_t$ is continuous. $T_1g$ is solution of the problem:
		\[\left\lbrace\begin{split}
		&\partial_t f+v\cdot\nabla_x f = 0\\
		&f_{|\{0\}\times \Lambda\times \mathbb{R}^3} = g\\
		&f_{\mathbb{R}_+\times\Sigma^+}=g_t
		\end{split}\right.\]
		where we fix all the boundary condition. Because $g$ is continuous outside  $(\mathcal{F}(t_1)\cup\cdots\cup\mathcal{F}(t_r))\cap\{t=0\}$ and $g_t$ is continuous outside  $\mathcal{F}(t_1)\cup\cdots\cup\mathcal{F}(t_r)$, $T^0_1g$ is continuous outside  $\mathcal{F}(t_0)\cup\cdots\cup\mathcal{F}(t_r)$.
		
		The bound follows from maximum principle.
	\end{proof}
	
	Let $f_1,\cdots,f_s$ be measurable functions, continuous outside $\mathcal{F}(t_1)\cup\cdots\cup\mathcal{F}(t_r)$ and dominated by $Ce^{-\beta|v|^2/2}$. Then $T^0_s(t)\left(f_1\otimes\cdots\otimes f_s\right)$ is equal to $ (T^0_1(t)f_1)\otimes\cdots\otimes(T^0_1(t)f_s)$ and
	\[C_{s-1,s}^0\left(f_1\otimes\cdots\otimes f_s\right) = \sum_{i = 1}^{s-1} f_1\otimes\cdots\otimes C_{1,2}^0 \left(f_i\otimes f_s\right)\cdots\otimes f_{s-1}.\]
	
	If we fix $t$ and $x$, $f_i$ has a discontinuity for  $v \in \bigcup_{i=1}^r\left\lbrace u\in\mathbb{R}^3,~(t,x,u)\in\mathcal{F}(t_i)\right\rbrace$ the set of velocities such that the particles meet the boundary at time $t_i$, which is a union of planes of dimension $2$. Using the "Calerman's collision parametrization" (see Appendix \ref{coordonné de Calerman}),
	\[\begin{split}C_{1,2}^0 f_1\otimes f_2(t,x,v) &:=
	- f_1(t,x,v)\int_{\mathbb{R}^2\times\mathbb{S}^{2}} f_2(t,x,v_*) ((v-v_*)\cdot\sigma)_- dv_*d\sigma \\
	&~~+ \int_{\mathbb{R}^3\times\mathbb{S}^{2}} f_1(t,x,v')f_2(t,x,v'_*) ((v-v_*)\cdot\sigma)_- dv_*d\sigma\\
	&:=
	- f_1(t,x,v)\int_{\mathbb{R}^2\times\mathbb{S}^{2}} f_2(t,x,v_*) ((v-v_*)\cdot\sigma)_- dv_*d\sigma \\
	&~~+ \int_{\{(v',v_*')\in\mathbb{R}^6,~(v'-v)\cdot(v'_*-v)=0\}} f_1(t,x,v')f_2(t,x,v'_*)  dv'dS(v'_*)
	\end{split}\]
 	where $dS(v'_*)$ is the Lebesgue measure on the affine plane $\{v_*'\in\mathbb{R}^3,~(v'-v)\cdot(v_*-v)=0\}$. In the first term of the sum, we integrate $f_2$ on a space of dimension $3$, so the set of discontinuity is of zero measure, so the first term is continuous outside $\mathcal{F}(t_1)\cup\cdots\cup\mathcal{F}(t_r)$. For the second term, because we integrate $v'$ on the full space, for almost all $v'$, $f_1(t,x,\cdot)$ is continuous at $v'$ and $v'_*$ lives in a plane transverse to the set of discontinuities of $f_2(t,x,\cdot)$. Thus $f_1(t,x,v')f_2(t,x,v'_*)$ is integrable and the second term is continuous.
	
	In addition we have the following bound:
	\begin{equation}\label{continuité de C1,2}
	\begin{split}\left\vert C_{1,2}^0 f_1\otimes f_2(t,x,v)\right\vert&\leq C^2 \int \left(e^{\frac{-\beta\left(|v|^2+|v_*|^2\right)}{2}} +e^{\frac{-\beta\left(|v'|^2+|v_*'|^2\right)}{2}}\right) ((v-v_*)\cdot\sigma)_-d\sigma dv_*\\
	&\leq AC^2 \int e^{\frac{-\beta\left(|v|^2+|v_*|^2\right)}{2}} (|v|+|v_*|)dv_* \\
	&\leq C^2 A \beta^{-3/2}\left(\beta^{-1/2} +|v|\right)e^{-\beta |v|^2/2}
	\end{split}\end{equation}
	for some constant $A$. Then for all  $\beta'<\beta$, we can bound $C_{1,2}^0f_1\otimes f_2$ by $\tilde{C}e^{-\beta' |v|^2/2}$. Finally we always have integrability  with respect to $v$. 
	
	To summarize, if $f_1,\cdots,f_{s+1}$ are measurable functions, continuous outside $(\mathcal{F}(t_1)\cup\cdots\cup\mathcal{F}(t_r))\cap\{t=0\}$ and dominated by $Ce^{-\beta|v|^2/2}$, then $C_{s,s+1}^0T_{s+1}^0(t)(f_1\otimes\cdots \otimes f_{s+1}) = \sum_{i=1}^s f_1^i\otimes\cdots \otimes f_{s}^i$ where the $(f_j^i)_{i,j}$ are continuous outside $\mathcal{F}(0)\cup\mathcal{F}(t_1)\cup\cdots\cup\mathcal{F}(t_r)$ and dominated by $\tilde{C}e^{-\beta'|v|^2/2}$ for all $\beta'<\beta$.
	
	Iterating the process  \[(t_{s+1},\cdots,t_{s+r})\mapsto T^0_s(t-t_{s+1})C^0_{s,s+1}\cdots C^0_{s+r-1,s+r}T^0_{s+r}(t_{s+r})f^{\otimes (s+r)}\] is equal to  
	$\sum_i f_1^i\otimes\cdots \otimes f_{s}^i$ where the $(f_j^i)_{i,j}$ are continuous outside $\mathcal{F}(0)\cup\mathcal{F}(t_{s+1})\cup\cdots\cup\mathcal{F}(t_{s+r})$. Fixing $(t,\gr{x}_s,\gr{v}_s)$, outside a finite number of $(t_{s+1},\cdots,t_{s+r})$, the $f_j^i$ are continuous near $(t_m,x_n,v_n)_{\tiny\begin{matrix}m\in\{1,\cdots,r\}\\n\in\{1,\cdots,s\}\end{matrix}}$. By integrating we get that $Q^0_{s,s+r}(t)f^{\otimes (s+r)}$ is well defined on $(\Lambda\times\mathbb{R}^3)^s$ and continuous outside $\mathcal{F}(0)$.
	
	\begin{remark}
	Note that if $f_0$ verifies the boundary condition
	\begin{equation}\label{compability condition}
	\forall (x,v)\in \Sigma^+,~f_0\left(x,v\right) = \int_{\mathbb{S}^2} f_0\left(x,\omega |v|\right)c_3\left(\Gamma(x) \omega\cdot\e\right)_-d\omega,
	\end{equation}
	$T^0_1 f_0$ is continuous, and then $Q_{1,r}^0f_0$ is always continuous. Finally we get:
	\begin{prop}
		Let $f_0$ be a continuous function on $\Lambda\times\mathbb{R}^3$, bounded by a Gaussian distribution and satisfying the condition \eqref{compability condition}. Then the solution $f(t)$ of the Boltzmann equation with diffusion in angle is continuous.
	\end{prop}
	\end{remark}
	\subsection{Continuity estimates}
	\begin{prop}
		There is a constant $C$ independent of $\beta$, $\mu$ and $r\in\mathbb{N}$ such that
		\begin{equation}\label{estimation de Q^0}
		\left\Vert Q_{s,s+r}^0(t)f^{s+r}_0\right\Vert_{3\beta/4,\mu-1}\leq \left(C \beta^{-2}e^{-\mu} t\right)^r \|f_{s+r}^0\|_{\beta,\mu},
		\end{equation}
		\begin{equation}\label{estimation de Q^epsilon}
		\left\Vert Q_{s,s+r}^\epsilon(t)f^{(s+r)}_{0,N}\right\Vert_{3\beta/4,\mu-1}\leq \left(C \beta^{-2}e^{-\mu}t\right)^r \|f^{(s)}_{0,N}\|_{\beta,\mu}.
		\end{equation}
	\end{prop}
	\begin{proof}
	First note that $\exp\left({-\beta\|\gr{v}_s\|^2/2}\right)$ is preserved by $T_s^0(t)$ for all $\beta>0$. Using the bound \eqref{continuité de C1,2}, if $f^{s+1}_0$ is dominated by $\exp\left(-\beta\|\gr{v}_{s+1}\|^2/2\right)$,
	\[\begin{split}
	\big\vert &C_{s,s+1}T_{s+1}(t)f^{s+1}_0(\gr{z}_{s+1})\big\vert\\
	&\leq C   \beta^{-3/2}\left(s\beta^{-1/2}+\sum_{i=1}^s |v_i|\right)e^{ -\frac{\beta}{2}\|\gr{v}_s\|^2}\\
	&\leq C   \beta^{-3/2}\left(s\beta^{-1/2}+\left(\sum_{i=1}^s |v_i|\right)e^{-(\beta-\beta')\|\gr{v}_s\|^2}\right)e^{ -\frac{\beta'}{2}\|\gr{v}_s\|^2}\\
	&\leq C   \beta^{-3/2}\left(s\beta^{-1/2}+ s^{1/2}(\beta-\beta')^{-1/2}\left(\sum_{i=1}^s (\beta-\beta')|v_i|^2\right)^{1/2}e^{-\frac{\beta-\beta'}{2}\sum_{i=1}^s|v_i|^2}\right) e^{-\frac{\beta'}{2}\|\gr{v}_s\|^2}\\
	&\leq C \beta^{3/2}\left(s\beta^{-1/2}+ s^{1/2}(\beta-\beta')^{-1/2}\right) e^{ -\frac{\beta'}{2}\|\gr{v}_s\|^2}
	\end{split}\]
	using the Cauchy-Schwartz inequality between lines 3 and 4. Then iterating this bound and integrating on $t>t_1>\cdots>t_r>0$, we get that for $\|(f_0^s)\|_{\beta,\mu}\leq 1 $
	\[\left\vert Q^0_{s,s+r}(t)f_0^{s+r}(\gr{z}_{s})\right\vert
		\leq \frac{t^r}{r!}\prod_{i=0}^{r-1} C \beta_{i+1}^{-3/2}\left( (s+i)\beta_{i+1}^{-1/2} +(s+i)^{1/2}(\beta_{i+1}-\beta_i)^{-1/2}\right) e^{ -\frac{3\beta}{8}\|\gr{v}_s\|^2-\mu (r+s)}\]
	for $\beta = \beta_r>\cdots>\beta_0 = 3\beta/4$. For $\beta_{i+1}-\beta_i = \beta/4r$, we get :
	\[\begin{split}
	\left\vert Q^0_{s,s+r}(t)f_0^{s+r}(\gr{z}_{s})\right\vert
	&\leq \frac{t^r(s+r)^{r}}{r!}\left(C\beta^{-2}\right)^re^{ -\frac{3\beta}{8}\|\gr{v}_s\|^2-\mu (r+s)}\\
	&\leq t^r e^{r+s}\left(C\beta^{-2}\right)^re^{ -\frac{3\beta}{8}\|\gr{v}_s\|^2-s\mu}e^{-\mu r}
	\end{split}\] 
	using the Stirling's formula. This is the expected estimate. 
	
	The same method works for $Q^\epsilon_{s,s+r}(t)$.
	\end{proof}
	We can sum these bounds and get
	\[\left\Vert f^{s}(t)\right\Vert_{3\beta/4,\mu-1}\leq \frac{1}{1-C \beta^{-2}e^{-\mu}t} \|(f^{s}_0)\|_{\beta,\mu}\]
	which concludes the proof of Theorem \ref{existence de la hierarchie de Boltzmann}.\qed

	\section{Main theorem}\label{Main theorem and strategy of the proof}
	The main theorem of the paper is a weak convergence result. Indeed we will only look at the convergence of  \emph{observables} of the system, \textit{i.e.} averaging with respect to the momentum variable. In addition, the marginals of the hard sphere system are only well defined in $\mathcal{D}^s_\epsilon$, so convergence will occur only away from the diagonal of the physical space. To summarize all these conditions, we define the following notion of convergence:
	\begin{definition}\label{definition de la convergence}
		A sequence $(f_N^s)_{1\leq s \leq N}$ converges to a sequence $(f^s)_s$ in average and locally uniformly off the diagonal if for all $\varphi_s : \mathbb{R}^3\rightarrow \mathbb{R}$, continuous with polynomial growth at infinity, and for all compact sets  $K\subset \Lambda^s$, away from the diagonal set $\mathfrak{D}^s :=\left\lbrace \gr{x_s} \in \Lambda^s|\exists i,j,~ x_i=x_j\right\rbrace$
		\begin{equation}
			I_{\varphi_s}(f_N^s-f^s)(\gr{x}_s) := \int_{\left(\mathbb{R}^3\right)^s} \left(f_N^s-f^s\right)(\gr{z}_s) \varphi_s(\gr{v}_s)d\gr{v}_s \rightarrow 0 \mathrm{~in~}L^\infty(K).
		\end{equation}
	\end{definition}
	
	\begin{theorem}\label{theorem de Landford}
		Let $\beta>0$ and $\mu$ be two constants. Then there exists a time $T$ such that the following holds. Let $(f_0^s)_s$ be an element of $\tilde{X}_{\beta,\mu}$ and $(W_{0,N})_N$ a sequence of symmetric functions in $L^1\cap L^\infty(\mathcal{D}^N_{1/{\sqrt{N}}})$ with $(f_{0,N}^{(s)})_{1\leq s\leq N}$ their marginals. Then if $(f_{0,N}^{(s)})$ converge to $(f_0^s)$ in norm $\|~\|_{\beta,\mu}$, the sequence $(f_N^{(s)}(t))_s$ solution to the BBGKY hierarchy converges in average locally uniformly to $(f^s(t))_s$ the solution of the Boltzmann hierarchy with initial data $(f^s_0)$, $\forall t<T$.
	\end{theorem}
	\begin{remark}
		Note that the convergence of the observables is uniform for $s=1$.
	\end{remark}
	\begin{proof}
	As for the interacting case, we  rewrite $Q^0_{s,s+r}(t)$ with a characteristic formula. We construct the \emph{pseudotrajectory} $\zeta^0( \tau, t,\gr{z}_s, \gr{(t,\nu,\bar{v})}_{s+1,s+r}, \gr{\bar{\omega}}_{r+s} ,a,\gr{\sigma})$ for $\tau\in[0,t]$, $\gr{(t,\nu,\bar{v})}_{s+1,s+r} := (t_i,\nu_i,\bar{v}_i)_{s<i\leq s+r}\in (\mathbb{R}\times\mathbb{S}^2\times\mathbb{R}^3)^r$  with $t>t_{s+1}\cdots>t_{s+r}>0$, $\gr{\bar{\omega}}_{r+s}=(\bar{\omega}_1,\cdots,\bar{\omega}_{r+s})\in\Omega^{r+s}$, $a\in\mathfrak{A}_s^{s+r}$ and $\gr{\sigma} \in \{\pm1\}^r$. The number of particles of $\zeta^ 0(\tau)$ is not constant: for $\tau$ between $t_k$ and $t_{k+1}$ there are $s+k$ particles (by convention, $t_s:=t$ and $t_{r+s+1}=0$). We define $(\zeta^0(\tau),\gr{\bar{\omega}}^0_{r+s}(\tau)):=(x^0_1(\tau),v_1^0(\tau),\cdots,x^0_{s+r(\tau)}(\tau),v_{s+r(\tau)}^0(\tau),\gr{\bar{\omega}}^0_{r+s}(\tau))$ by
	\begin{itemize}
		\item $(\zeta^{ 0}(t),\gr{\bar{\omega}}_{r+s}^ 0(t)) := (\gr{z}_s,\gr{\bar{\omega}}_{r+s})$
		\item  For $\tau \in(t_{k+1},t_k)$, $\gr{\bar{\omega}}^ 0_{k+1,r+s}(\tau)$ is constant and
		\[(\zeta_{1,k}^ 0(\tau),\gr{\bar{\omega}}^0_{1,k}(\tau))=\phi^{ 0,(\tau-t_k)}_{k}\left(\zeta_{1,k}^ 0(t_k^+), \gr{\bar{\omega}}^ 0_{1,k}(t_k^+)\right)\]
		where $\phi^{ 0,(\tau-t_k)}_{k}$ is the flow of punctual particles (particles do not see each other).
		\item At time $t_k^+$, a particle is added at the position  $x^0_{k}=x^0_{a(k)}(t_k)$. If $\sigma_k =1$, we will have $\nu_k\cdot(\bar{v}_{k}-v^0_{a(k)}(t_k^+))>0$ and the velocities $(v^0_{a(k)}(t_k^-),{v}^0_{k}(t_k^-))$ are given by the usual scattering
		\[\left\lbrace\begin{split}
		&v^0_{a(k)}(t_k^-) = v^0_{a(k)}(t^+_k) - \nu_k\cdot\left(v^0_{a(k)}(t^+_k)-\bar{v}_{k}\right)\nu_k\\
		&v_{k}^0(t_k^-)= \bar{v}_{k}+\nu_k\cdot\left(v_{a(k)}^0(t^+_k)-\bar{v}_{k}\right)\nu_k~.
		\end{split}\right. \] 
		Else if $\sigma_k =-1$, $\nu_k\cdot(\bar{v}_{k}-v_{a(k)}^0(t_k^+))<0$ and we will have no scattering. The velocities are at time $t_k^-$ just $(v_{a_{k}}^0(t_k^+),\bar{v}_{k})$.
	\end{itemize}
	
	We denote $\mathcal{G}^0(\gr{z}_s,t,a,\gr{\sigma})$ the set of coordinates, \textit{i.e.} the $(\gr{(t,\nu,\bar{v})}_{s+1,s+r},\gr{\bar{\omega}}_{r+s})$ such that the pseudo-trajectory follows the collision parameters $(a,\gr{\sigma})$, and 
	\[d\Lambda^0_{a,\gr{\sigma}}(\gr{(t,\nu,\bar{v})}_{s+1,s+r},\gr{\bar{\omega}}_{s+r}) = \left(\prod_{k=s+1}^{s+r}\left[\sigma_k\nu_k\cdot\left(\bar{v}_k-v^0_{a(k)}\left(t_k^+\right)\right)\right]_+
	dt_k d\bar{v}_{k}d\nu_k\right) d\mathbb{P}^{s+r}(\bar{\gr{\omega}}_{r+s}).\]
	
	Then we can rewrite $Q^0_{s,s+r}(t)f^{s+r}(\gr{z}_s)$ as
	\begin{equation}
	Q_{s,s+r}^0(t)f^{s+r}(\gr{z}_s) := \sum_{a,\gr{\sigma}} \sigma_{s+1}\cdots\sigma_{s+r}\int_{\mathcal{G}^0(\gr{z}_s,t,a,\gr{\sigma}_r)} d\Lambda^0_{a,\gr{\sigma}_r}(\gr{(t,\nu,\bar{v})}_{s+1,s+r},\gr{\bar{\omega}}_{s+r}) f_{0}^{s+r}(\zeta^0(0)).
	\end{equation}
	
	We shall prove the one-to-one convergence 
	\[I_{\varphi_s}\left(\int_{\mathcal{G}^\epsilon(\gr{z}_s,t,a,\gr{\sigma})} d\Lambda_{a,\gr{\sigma}}^\epsilon f_{0,N}^{(s+r)}(\zeta^\epsilon(0))\right)\rightarrow I_{\varphi_s}\left(\int_{\mathcal{G}^0(\gr{z}_s,t,a,\gr{\sigma})} d\Lambda_{a,\gr{\sigma}}^0 f_{0}^{s+r}(\zeta^0(0))\right)\]
	at all points $(t,\gr{x}_s)$ and all collision parameters $(a,\gr{\sigma}_r)$, using that the pseudo trajectories $\zeta^\epsilon$ and $\zeta^0$ are in average not so far. There are two reasons producing a big error between the two trajectories. The first one is the \emph{recollisions} (collision of two particles that have already been created) that can occur for the interacting process but not for the punctual particles. The second comes from the difference between $\bar{\gr{\omega}}^\epsilon(0)$ and $\bar{\gr{\omega}}^0(0)$: because there is a shift between $\zeta^0(\tau)$ and $\zeta^\epsilon(\tau)$, the times of reflection are not exactly the same. Thus there exists a set of times during which $\omega^\epsilon(\tau)$ is different from $\omega^0(\tau)$. If $0$ is in this set the final velocity will be very different.
	
	To make the convergence term by term work, we have to perform some truncation. First we note that the continuity bound gives
	\begin{equation}
	\mathcal{R}_1:=\left\vert \sum_{r=0}^{N-s} Q_{s,s+r}^\epsilon (f_{0,N}^{(s+r)}-f_0^{s+r})\right\vert\leq C_{\beta,\mu} e^{-\frac{3\beta}{8}\|\gr{v}_s\|^2 -(\mu-1)s}\left\Vert\left(f_{0,N}^{(s+r)}-f_0^{s+r}\right)_{1\leq s\leq N} \right\Vert_{\beta,\mu}
	\end{equation}
	and we have only to look at the semigroups applied to $(f_0^s)_s$.
	
	Then we have to bound the number of creations of particles by an integer $R$ that depends on $N$. The rest will be bounded by :
	\begin{multline}\label{borne sur le nombre de création}
	\mathcal{R}_2:=\left\vert \sum_{r> R} Q_{s,s+r}^0(t) f^{s+r}_0\right\vert \leq  \sum_{r> R} (C_{\beta,\mu}t)^r e^{-\frac{3\beta}{8}\|\gr{v}_s\|^2 -(\mu-1)s}\|(f_0^s)_s\|_{\beta,\mu}\\
	\leq C2^{-R}e^{-\frac{3\beta}{8}\|\gr{v}_s\|^2 -(\mu-1)s}\|(f_0^s)_s\|_{\beta,\mu}
	\end{multline}
	for $t\leq 1/(2C_{\beta,\mu})=:T$. The same estimates hold for the BBGKY hierarchy. 
	
	Next we have to cut the high energy. Because $(f^s_0)_s$ is bounded in $X_{\beta,\mu}$, $(f^s_0\ind_{\|\gr{v}_i\|^2>E})_s$ is bounded in $X_{7\beta/8,\mu}$ by $e^{-7\beta E/16}\|(f_0^s)_s\|_{\beta,\mu}$. Thus we have the bound (the same bound holds for BBGKY)
	\begin{equation}\label{borne sur les grandes énergies}
	\mathcal{R}_3:= \left\vert \sum_{r\geq 0} Q^0_{s,s+r}(t) \left(f_0^{s+r}\ind_{\|\gr{v}_{s+r}\|^2>E}\right)\right\vert\leq Ce^{-\beta E/16} e^{-\frac{3\beta}{8}\|\gr{v}_s\|^2 -(\mu-1)s}\|(f_0^s)_s\|_{\beta,\mu}.
	\end{equation}
	
	Thus for any $\eta>0$ we can fix $R$, $E$ and $N_0$ such that for $N>N_0$, 
	\[\left\Vert\mathcal{R}_1+\mathcal{R}_2+\mathcal{R}_3\right\Vert_{3\beta/4,\mu-1}\leq \eta.\]
	From this  $\mathcal{R}_1+\mathcal{R}_2+\mathcal{R}_3$ converges to zero in the sense of Definition \ref{definition de la convergence} when $R$, $E$ and $N_0$ tend to infinity, uniformly in $\epsilon$.
	
	Therefore the proof is reduced to the proof of the following proposition: 
	\begin{prop}\label{Prop de convergence term a term}
		For all compact sets $K\subset \Lambda^s\setminus \mathfrak{D}_s$, all collision parameters $(a,\gr\sigma)$ and all test functions $\varphi_s$, 
		\begin{equation}\label{convergence term à term}
		\int_{\mathbb{G}^\epsilon(t,\gr{x}_s,a,\gr{\sigma},E)} d\gr{v}_sd\Lambda_{a,\gr\sigma}^\epsilon f_{0}^{s+r}(\zeta^\epsilon(0))\varphi_s(\gr{v}_s) \rightarrow \int_{\mathbb{G}^0(t,\gr{x}_s,a,\gr{\sigma},E)} d\gr{v}_sd\Lambda_{a,\gr\sigma}^0 f_{0}^{s+r}(\zeta^0(0))\varphi_s(\gr{v}_s)
		\end{equation}
		uniformly on $[0,T]\times K$, where $\mathbb{G}^\epsilon(t,\gr{x}_s,a,\gr{\sigma},E)$ is the set of parameters 
		\begin{multline*}\Big\lbrace (\gr{v}_s,\gr{(t,\nu,\bar{v})}_{s+1,s+r},\gr{\bar{\omega}}_{s+r}),~\gr{v_s}\in\mathbb{R}^{3s},~(\gr{(t,\nu,\bar{v})}_{s+1,s+r},\gr{\bar{\omega}}_{s+r})\in\mathcal{G}^\epsilon(t,\gr{x}_s,\gr{v}_s,a,\gr{\sigma}),\\
		\|\gr{v}_s\|^2+\|\gr{\bar{v}}_{s+1,s+r}\|^2\leq E\Big\rbrace,\end{multline*}
		and same definition for $\mathbb{G}^0(t,\gr{x}_s,a,\gr{\sigma},E)$.
	\end{prop}
	\end{proof}
	\section{Proof of proposition \ref{Prop de convergence term a term}}\label{Term by term convergence}	
	We fix collision parameters $(a,\gr{\sigma})$ and a compact set $K\in\Lambda^s\setminus\mathfrak{D}^s$. In order to prove \eqref{convergence term à term}, we compare the two pseudotrajectories. Fixed an initial position $\gr{x}_s$ at time $t$, we need to know the first difference between hard spheres and punctual process (going backward from time $t$ to $0$). We can construct four bad sets $ \mathcal{P}_1(\epsilon)$, $ \mathcal{P}_2(\epsilon)$, $ \mathcal{P}_2'(\epsilon)$ and $\mathcal{P}_3(\epsilon)$ such that
	\begin{itemize}
		\item for $(\gr{v}_s,(\gr{t,\nu,\bar{v}})_{s+1,s+r},\gr{\bar{\omega}}_{s+r})\in\mathcal{P}_1(\epsilon)$,  there is first a {\it shift}, \textit{i.e.} there is a $i\in\{s+1,\cdots,s+r\}$ such that $\bar{\gr{\omega}}_{s+r}^0(t_i)\neq \bar{\gr{\omega}}_{s+r}^\epsilon(t_i)$ and there is no recollision during $(t_i,t)$ (recall that $t_{s+r}=0$),
		\item for $(\gr{v}_s,(\gr{t,\nu,\bar{v}})_{s+1,s+r},\gr{\bar{\omega}}_{s+r})\in \mathcal{P}_2(\epsilon)$, the hard sphere process has first a recollision, \textit{i.e.} there is a first recollision at time $\tau$ and no shift during $(\tau,t)$,
		\item for $(\gr{v}_s,(\gr{t,\nu,\bar{v}})_{s+1,s+r},\gr{\bar{\omega}}_{s+r})\in\mathcal{P}_2'(\epsilon)$, the punctual process has first an \emph{overlap}, \textit{i.e.} two particles reach a distance less than $\epsilon$ at some time $\tau$ (with an exception for particles $i$ and $a(i)$ just after $t_i$),
		\item for $(\gr{v}_s,(\gr{t,\nu,\bar{v}})_{s+1,s+r},\gr{\bar{\omega}}_{r+s})\in\mathcal{P}_3(\epsilon)$, there is some $i\in\{s+1,\cdots,s+r\}$ such that $x^\epsilon_{a(i)}(t_i)$ is at distance less than $2\epsilon$ from the boundary or from any other particle, and $\mathcal{P}_3(\epsilon)\cap\left(\mathcal{P}_1(\epsilon)\cup\mathcal{P}_2(\epsilon)\cup\mathcal{P}'_2(\epsilon)\right)=\emptyset$.
	\end{itemize}
	
	Note that $\mathbb{G}^\epsilon\setminus\left( \mathcal{P}_1(\epsilon)\cup \mathcal{P}_2(\epsilon)\cup \mathcal{P}'_2(\epsilon)\cup \mathcal{P}_3(\epsilon)\right) = \mathbb{G}^0\setminus\left( \mathcal{P}_1(\epsilon)\cup \mathcal{P}_2(\epsilon)\cup \mathcal{P}'_2(\epsilon)\cup \mathcal{P}_3(\epsilon)\right)$. Indeed outside $\mathcal{P}_3$, all relative positions $\nu_i$ are allowed at creation $i$, and outside $ \mathcal{P}_1(\epsilon)\cup \mathcal{P}_2(\epsilon)\cup \mathcal{P}'_2(\epsilon)$ the velocities of the incoming particles at the reflections are the same in hard sphere and punctual backward pseudotrajectories.

	In the following we will need to integrate on the collision parameters, we then introduce the following useful terminology. The particle $a(i)$ is \emph{deviated} by the creation $i\in\{s+1,\cdots s+r\}$ if $\sigma_i$ is positive. In that case there is a scattering of the speed. The particle $i$ is by definition deviated by the creation $i$. We recall that we look at backward trajectories and we follow time in the inverse sense.
	\subsection{Continuity estimates of the two pseudotrajectories}
	We begin by a continuity estimate for the process involving one particle.
	
	\begin{lemma}
		Let $(x_0,\tilde{x}_0)\in\Lambda^2$, $v\in B(E^{1/2})$ (the ball in $\mathbb{R}^3$ of diameter $E^{1/2}$) and $\gr{\bar{\omega}}_2\in\Omega^2$. The coordinates $(x_0,v)$ (respectively $(\tilde{x}_0,v)$) correspond to a particle $1$ in the punctual process (respectively in hard sphere process). At time $t$ we create a particle $2$: for $(\nu,v_*)\in\mathbb{S}^2\times B(E^{1/2})$, $\tau\in(0,t)$,   the two punctual  backward pseudotrajectories are defined as follow: for $\tau \in (0,t)$
		\begin{itemize}
			\item $(\gr{z}_2(\tau),\gr{\bar{\omega}}_2(\tau))$ is the backward trajectory without interaction with initial conditions $(x_0, v,$ $x_0, v_*, \gr{\bar{\omega}}_2)$ if $(v_*-v)\cdot\nu>0$,  and $(x_0 ,v', x_0, v_*', \gr{\bar{\omega}}_2)$ else,
			\item $(\tilde{\gr{z}}_2(\tau),\gr{\tilde{\bar{\omega}}}_2(\tau))$
			 is the backward trajectory without interaction with initial conditions $(\tilde{x}_0,v,$ $\tilde{x}_0+\epsilon\nu,v_*,\gr{\bar{\omega}}_2)$ if $(v_*-v)\cdot\nu>0$,  and $(\tilde{x}_0,v', \tilde{x}_0+\epsilon\nu,v_*',\gr{\bar{\omega}}_2)$ else.
		\end{itemize}
		
		Then if the particle $i\in\{1,2\}$ is deviated by the creation, there exists $C>0$ such that for $\tau\in(0,t)$
		\begin{multline}
		\int_{B(E^{1/2})\times\mathbb{S}^2} |x_i(\tau)-\tilde{x}_i(\tau)|\ind_{\bar{\omega}_i(\tau)=\tilde{\bar{\omega}}_i(\tau)} |\nu\cdot(v-v_*)|dv_*d\nu \\ \leq CE^{5/2}\left(\epsilon+ |x_0-\tilde{x}_0|\right)\left|\log\left(\epsilon+ |x_0-\tilde{x}_0|\right)\right|.
		\end{multline}
	\end{lemma}
	\begin{proof}
		Fix $(v_*, \nu)$ and denote $\tau_k$ (respectively $\tau_k+\delta\tau_k$) the time of the $k$-th reflection of the particle $z_i$ (respectively the particle $\tilde{z}_i$). For $k\geq 1$, because between their $k$-th and $(k+1)$-th reflections the two particles have the same velocity and the two sides of $\partial\Lambda$ are parallel, \[(x_i(\tau_k),x_i(\tau_{k+1}),\tilde{x}_i(\tau_{k+1}+\delta\tau_{k+1}),\tilde{x}_i(\tau_{k}+\delta\tau_{k}))\] forms a parallelogram. Hence $\delta\tau_k$ and $(\tilde{x}_i(\tau_k+\delta\tau_k)-x_i(\tau_k))$ do not depend on $k$ (we denote them $\delta\tau$ and $\delta x$).
		
		Because we suppose that there is no shift at time $\tau$, $z_i$ and $\tilde{z}_i$ have the same number of reflections. If there is no reflection, $|x_i(\tau)-\tilde{x}_i(\tau)|$ is constant. If there are $k\geq 1$ reflections, we can suppose that $\delta\tau>0$ and
		\[\begin{split}
		\left|x_i(\tau)-\tilde{x}_i(\tau)\right|=\left|x_i(\tau_k)-\tilde{x}_i(\tau_k)\right|&\leq\left|x_i(\tau_k)-\tilde{x}_i(\tau_k+\delta\tau)\right|+\left|\tilde{x}_i(\tau_k+\delta\tau)-\tilde{x}_i(\tau_k)\right|\\
		&\leq |\delta x|+E^{1/2}|\delta\tau|.
		\end{split}\]
		We can treat similarly the case $\delta\tau<0$.
			
		Because  the two trajectories are parallel before the first reflection, 
		\[\delta \tau = ([x_i(0)-\tilde{x}_i(0)]\cdot \e)/(v_i\cdot\e) \mathrm{~and~}|\delta x| \leq \frac{|v_i|}{|v_i\cdot\e|}\left|x_i(0)-\tilde{x}_i(0)\right|.\]
		Thus if $\bar{\omega}_i(\tau) = \tilde{\bar{\omega}}_i(\tau)$, 
		\[|x_i(\tau)-\tilde{x}_i(\tau)|\leq|x_i(0)-\tilde{x}_i(0)|\left(1+\frac{2E^{1/2}}{|v_i\cdot\e|}\right)\leq\left(|x_0-\tilde{x}_0|+\epsilon\right)\frac{3E^{1/2}}{|v_i\cdot\e|},\]
		and it is also bounded by some constant $C_\Lambda$ because $\Lambda$ is compact. Because the particle $i$ is deviated, we can integrate on the scattering: using estimation \eqref{intégration de 1/v} (Appendix \ref{coordonné de Calerman})
		\[\begin{split}\int_{B(E^{1/2})\times\mathbb{S}^2} |x_i(\tau)-\tilde{x}_i(\tau)|&\ind_{\bar{\omega}_i(\tau)=\tilde{\bar{\omega}}_i(\tau)} |\nu\cdot(v-v_*)|dv_*d\nu\\
		&\leq\int_{B(E^{1/2})\times\mathbb{S}^2}\left(\frac{3E^{1/2}\left(|x_0-\tilde{x}_0|+\epsilon\right)}{|v_i\cdot\e|}\wedge C_\Lambda\right)|\nu\cdot(v-v_*)|dv_*d\nu\\
		&\leq CE^{5/2}(|x_0-\tilde{x}_0|+\epsilon)\log(|x_0-\tilde{x}_0|+\epsilon).\end{split}\]
	\end{proof}
	
	Fix $(t_{s+1},\cdots)$. Let $\tau\in[0,t]\setminus \{t_i,s<i<s+r\}$. If there is no recollision nor overlap and if $\gr{\bar{\omega}}^\epsilon(\hat{t})={\gr{\bar{\omega}}}^0(\hat{t})$ for $\hat{t}\in\mathfrak{T}(\tau):=\{\tau\}\cup\{t_i, t_i>\tau\}$, then by an iteration argument $v^0_k(\tau)=v^\epsilon_k(\tau)$ (in the following denoted $v_k(\tau)$). Fixing $\gr{x}_s$, $\gr{v}_s$, $\gr{t}_{s+1,s+r}$ and $\gr{\bar{\omega}}_{s+r}$, we can then iterate the previous estimates and obtain 
	
	\begin{prop}
		For $k\in\{1,\cdots,r\}$, $\tau\in[t_{k+1},t_k)$ and $i\in\{1,\cdots,s+k\}$ fixed, we have that:
		\begin{multline}\label{estimation de continuité du flot}
		\int_{B_k(E^{1/2})\times(\mathbb{S}^2)^k}\ind_{\mathrm{no~recollision}}\ind_{\bar{\gr{\omega}}^\epsilon(\hat{t})=\bar{\gr{\omega}}^0(\hat{t}),~\hat{t}\in\mathfrak{T}(\tau)} |x_i^0(\tau)-{x}^\epsilon_i(\tau)|\times \prod_{l=s+1}^{s+k}\left|\nu_l\cdot\left(\bar{v}_l-v_{a(l)}\left(t_l^+\right)\right)\right|d\nu_ld\bar{v}_l
		\\
		\leq \epsilon k\left(CE^{5/2}|\log\epsilon|\right)^{k}
		\end{multline}
		for $B_k$ the ball on $\mathbb{R}^{3k}$.
	\end{prop}
	\begin{proof}
		Note that $B_k(E^{1/2})\subset(B_1(E^{1/2}))^k$, thus we will replace the first one by the second.
		
		In the following $\ti$ will indicate the following pseudoparticle: $\ti$ is equal to $i$ from time $\tau$ to time $t_i$. Then $\ti$ becomes $a(i)$ until time $t_{a(i)}$, \textit{etc} until $t$. We can suppose that every creation deviates $\ti$. Else it does not influence $|x_i^0(\tau)-{x}^\epsilon_i(\tau)|$ and we can count it as a factor $CE^2$.
		
		We want to prove recursively for $m\in\{1,\cdots k\}$
		\begin{multline}\label{borne de continuité, récurence}
		\int_{(B(E^{1/2})\times\mathbb{S}^2)^m}\ind_{\scriptscriptstyle{\mathrm{no~recollision}}}\ind_{\begin{matrix}\scriptscriptstyle{\bar{\gr{\omega}}^\epsilon(\hat{t})=\bar{\gr{\omega}}^0(\hat{t}),}\\\scriptscriptstyle{\hat{t}\in\mathfrak{T}(\tau)}\end{matrix}} |x_i^0(\tau)-{x}^\epsilon_i(\tau)| \prod_{l=s+k+1-m}^{s+k}\left|\nu_l\cdot\left(\bar{v}_l-v_{a(l)}\left(t_l^+\right)\right)\right|d\nu_ld\bar{v}_l
		\\
		\leq \left(CE^{5/2}\right)^m |\log\epsilon|^{m-1}\left[(m-1)\epsilon+\eta\left(\left|x^0_\ti(t_{s+k+1-m}^+)-x^\epsilon_\ti(t_{s+k+1-m}^+)\right|+\epsilon\right)\right]
		\end{multline}
		for $\eta(x)=x|\log(x)|$.
		The initialization is provided by the previous lemma. To prove the induction, as in the previous lemma
		\[\left|x^0_\ti(t_{s+k+1-m}^+)-x^\epsilon_\ti(t_{s+k+1-m}^+)\right|\leq\left( \left|x^0_\ti(t_{s+k-m}^-)-x^\epsilon_\ti(t_{s+k-m}^-)\right|\frac{3E^{1/2}}{\left|v_\ti(t_{s+k-m}^-)\cdot\e\right|}\right)\wedge C_\Lambda.\]
		Thus $\left(\left|x^0_\ti(t_{s+k+1-m}^+)-x^\epsilon_\ti(t_{s+k+1-m}^+)\right|+\epsilon\right)$ stays between $\epsilon$ and $C_\Lambda+\epsilon$ and its $\log$ is smaller than $\sup(|\log\epsilon|,|\log C_\Lambda+\epsilon|)= |\log\epsilon|$. Hence the left hand side of \eqref{borne de continuité, récurence} is bounded by
		\begin{equation*}
		\left(CE^{5/2}\right)^m |\log\epsilon|^{m}\left[m\epsilon+\left(\left|x^0_\ti(t_{s+k-m}^-)-x^\epsilon_\ti(t_{s+k-m}^-)\right|\frac{3E^{1/2}}{\left|v_\ti(t_{s+k-m}^-)\cdot\e\right|}\right)\wedge C_\Lambda\right].
		\end{equation*}
		With help of \eqref{intégration de 1/v}, we can integrate this estimates on $B(E^{1/2})\times\mathbb{S}^2$ with respect to measure \[ \left|\nu_{s+k-m}\cdot(\bar{v}_{s+k-m}- v_{a(s+k-m)}(t_{s+k-m}^+))\right|d\nu_{s+k-m} d\bar{v}_{s+k-m}\]
		and obtain the expected estimates.
		
		Because before the first creation punctual and hard sphere backward processes coincide, $\left|x^0_\ti(t_{s+1}^+)-x^\epsilon_\ti(t_{s+1}^+)\right|$ vanishes. Therefore $m=k$ gives the expected result.
	\end{proof}

	Integrating \eqref{estimation de continuité du flot} on $(\gr{v}_s,\gr{t}_{s+1,s+r},\gr{\bar{\omega}}_{s+r})$, 
	\begin{equation}
	\int_{\mathbb{G}^0\setminus\left(\mathcal{P}_1\cup \mathcal{P}_2\cup\mathcal{P}_2'\cup\mathcal{P}_3\right)(\epsilon)} \left|\zeta^\epsilon(0)-\zeta^0(0)\right|d\gr{v}_sd\Lambda^0\leq \epsilon|\log\epsilon|^r \frac{\left(CE^{5/2}\right)^{s+r}T^r}{r!}.
	\end{equation}
	Note that if the domain is the torus, the previous proof provides the bound $\epsilon{\left(CE^{5/2}\right)^{s+r}T^r}/{r!}$. The factor $|\log\epsilon|^r$ originates from the grazing reflections, and the author does not believe that this continuity estimate can be improved considerably.
	
	Because $f_0^{s+r}$ is continuous and $(\Lambda\times B(E^{1/2}))^{s+r}$ is compact, $f_0^{s+r}$ is uniformly continuous and the error goes to zero. $\varphi_s(\gr{v}_s)$ is bounded on $\mathbb{G}^0$. Finally we get:
	\begin{prop}
	\begin{equation}
	\int_{\mathbb{G}^0\setminus\left(\mathcal{P}_1\cup \mathcal{P}_2\cup\mathcal{P}_2'\cup\mathcal{P}_3\right)(\epsilon) }\left[f_{0}^{s+r}(\zeta^\epsilon(0))-f_{0}^{s+r}(\zeta^0(0))\right]\varphi_s(\gr{v}_s)d\gr{v}_sd\Lambda^0 \rightarrow 0
	\end{equation}
	uniformly in $[0,T]\times K$.
	\end{prop}
	
	\subsection{Estimation of $\mathcal{P}_1$}
	We decompose $\mathcal{P}_1$ into sets $\mathcal{P}_1^{i,j,i'}$: for $(\gr{v}_s,(\gr{t,\nu,\bar{v}})_{s+1,s+r},\gr{\bar{\omega}}_{r+s})\in\mathcal{P}_3(\epsilon)\in\mathcal{P}_1^{i,j,i'}$, there is no shift nor recollision in the interval $(t_{i'},t)$, $\bar{\omega}^0_j(t_{i'}^-)\neq\bar{\omega}^\epsilon_j(t_{i'}^-)$ at time $t_{i'}$ and the last creation deviating particle $j$ is $i$. There are at most $(s+r)^3$ such subsets and we only have to control $|\mathcal{P}_1^{i,j,i'}|$.
	
	We denote in the following $k_\iota(\tau)$ the last reflection of the particle $\iota$ before time $\tau$. Then $k_j(t_{i'})$ depends only on $(\gr{v}_s,\gr{(t,\nu,\bar{v})}_{s+1,i})$, on the $((\bar{\omega}_\iota^\kappa)_{\kappa\geq k_\iota(t_i)})_{1\leq\iota\leq i-1}\subset \gr{\bar{\omega}}_{s+r}$, and on the $(\bar{\omega}_j^\kappa)_{k_j(t_i)>\kappa>k_j(t_{i'})}$. The shift at time $t_{i'}$ does not depend on the \emph{remaining parameters} (which contribution is bounded by some constant $C(E,T,s+r)$).
	
	Then we look at the set of $(t_{i'},(\omega_j^k)_{k<k_j(t_i)})$ such that there is a shift. We can bound it directly by $T$.
	
	We denote, for $k>k_j(t_i)$, $\tau_k$ (respectively $\tau_k+\delta\tau$) the time of the reflection of $\omega_j^k$ after $t_i$ in the punctual process (respectively in the hard sphere process). Because we restrict to the interval $(t_{i'},t_i)$, the particle $j$ has no more collisions until the shift. As explained in the previous section until there is a shift, $\delta\tau$  does not depend on $k$. Denoting $|v|:= |v_j(t_i^-)|$, the $\tau_k$ follow the recurrence law $\tau_{k-1} = \tau_k - {1}/{(|v| |\omega_j^k\cdot \e|)}$. Thus there is a shift only if for some $k\geq k_j(t_i)$, $t_{i'}$ stays between $\tau_{k+1}$ and $\tau_{k+1}+\delta\tau$, and thus $t_{i'}\in [\tau_{k+1}-|\delta\tau|,\tau_{k+1}+|\delta\tau|]$. In the case where the two particles have at least one reflection after $t_i$, $k<k_j(t_i)$ and 
	\[{|v|\left((\tau_{k}-t_{i'})-|\delta\tau|\right)}\leq\frac{1}{|\omega_j^k\cdot\e|}\leq {|v|\left((\tau_{k}-t_{i'})+|\delta\tau|\right)}\]
	where $v = v_j(t_i^-)$. There is a shift only if the left member is negative and the right one is positive, \textit{i.e.} $t_{i'}\in [\tau_k-|\delta\tau|,\tau_k+|\delta\tau|]$, or if $|\omega_j^k\cdot\e|$ stays in $\left[\frac{1}{{|v|\left((\tau_{k}-t_{i'})+|\delta\tau|\right)}},\frac{1}{{|v|\left((\tau_{k}-t_{i'})-|\delta\tau|\right)}}\right]$. Using that the surface of $\{\omega\in \mathbb{S}^2,~\omega\cdot \e\in[a,b]\}$ is smaller that $2\pi ((b-a)\wedge 2)$, 
	\[ \begin{split}
	\int_{[0,T]\times\mathbb{S}^2}& \left(\ind_{t_{i'}\in[\tau_k-|\delta\tau|,\tau_k+|\delta\tau|]}+\ind_{|\omega_j^k\cdot\e|\in\left[\frac{1}{|v|\left((\tau_{k}-t_{i'})+|\delta\tau|\right)},\frac{1}{|v|\left((\tau_{k}-t_{i'})-|\delta\tau|\right)}\right]}\right) (\omega^k_j\cdot\e)_+ dt_{i'}d\omega^k_j\\
	&\leq C|\delta\tau|+C\int\limits_0^T \left(\ind_{t_{i'}\notin[\tau_k-|\delta\tau|,\tau_k+|\delta\tau|]}{\frac{|\delta\tau|}{|v|(t_{i'}-\tau_k-|\delta\tau|)(t_{i'}-\tau_k+|\delta\tau|)}}\wedge 1\right)dt_{i'}\\
	&\leq C|\delta\tau|+{2C{|\delta\tau|}}\int\limits_1^{T/|\delta\tau|}{\frac{1}{|v|(s^2-1)}}\wedge 1ds~\\
	&\leq C|\delta\tau| + \frac{2C{|\delta\tau|}}{{|v|}}\leq C|\delta\tau|\left(1+\frac{2}{|v\cdot\e|}\right).
	\end{split}\]
	making the change of variable $s = \pm\frac{t_{i'}-\tau_k}{|\delta\tau|}$.
	
	In the case where only one particle has a reflection in $(t_{i'},t_i)$, $t_{i'}$ has to stay in $[\tau_1 -|\delta\tau|,\tau_1 +|\delta\tau|]$, and thus in a set of size $2|\delta\tau|$. We sum on all possible "last reflections". There are at most $E^{1/2}T$ reflections and as in the previous section $|\delta\tau|\leq (\epsilon+|x_{a(i)}^0(t_i)-x_{a(i)}^\epsilon(t_i)|)/(|v\cdot \e|)$. Finally the set of parameters $(t_{i'},(\omega_j^k)_{k<k_j(t_i)})$ such that there is a shift is of size at most 
	\[CT\left({\left(\epsilon+|x_{a(i)}^0(t_i)-x_{a(i)}^\epsilon(t_i)|\right)}\left(\frac{1}{|v\cdot\e|}+\frac{1}{|v\cdot\e|^2}\right)\wedge 1\right).\]
	Integrating over $(\nu_i,\bar{v}_i)$, and applying \eqref{intégration de 1/v} and \eqref{intégration de 1/v^2}, the set of parameters $((\nu_i,\bar{v}_i),t_{i'},$ $(\omega_j^k)_{k<k_j(t_i)})$ is of size at most 
	\[\left(\epsilon+|x_{a(i)}^0(t_i)-x_{a(i)}^\epsilon(t_i)|\right)^{1/2}C(E,T).\]
		
	Combining this estimation, the estimation of $|x_{a(i)}^0(t_i)-x_{a(i)}^\epsilon(t_i)|$ in \eqref{estimation de continuité du flot} and the estimation of the remaining term, $|\mathcal{P}^{i,j,i'}_1|$ converges uniformly to $0$.\qed
	
	\subsection{Estimation of $\mathcal{P}_2\cup\mathcal{P}_2'\cup\mathcal{P}_3$}
	In the following we look only at the hard sphere process and we drop the exponents $\epsilon$.
	
	First we estimate $|\mathcal{P}_2|$; the size of $\mathcal{P}_2'$ and $\mathcal{P}_3$ can be estimated similarly. 
	
	We begin by cutting the grazing velocities: we consider the set of initial parameters $\mathcal{P}_4$ such that
	\begin{itemize}
		\item for $i\in\{1,\cdots,s\}$, $|v_i\cdot\e|>\epsilon^{1/4}$
		\item for $i\in\{s+1,\cdots,s+r\}$, if the particle $j$ is deviated at time $t_i$, $|v_j(t^-_i)\cdot \e|>\epsilon^{1/4}$, $\forall k \in \{1,\cdots,i\}\setminus\{j\}$, $|(v_j(t_i^+)-v_k(t_i^+))\cdot\e|> \epsilon^{1/4}$,
		\item for any reflection of a particle $i$ at time $\tau$, the reflected velocity of reflection $v_i(\tau^-):= \omega_i^l |v_i(\tau^+)|$ has to verify $|v_i(\tau_i^-)\cdot \e|> \epsilon^{1/4}$, $\forall k\neq i$, $|v_j(\tau^-)-v_k(\tau^-)|> \epsilon^{1/4}$.
	\end{itemize}
	Because there are at most $(s+r)E^{1/2}T$ reflections in the interval $[0,t]$, the last condition deals only with a finite number of reflection parameters in $\gr{\bar{\omega}}_{s+r}$. Hence the size of the set $\mathbb{G}^\epsilon\setminus\mathcal{P}_4$ goes to zero. We restrict to $\mathcal{P}_4$ from now on. We split $\mathcal{P}_2\cap \mathcal{P}_4$ into the partition $(\mathcal{P}_{i,j})_{1\leq i < j \leq s+r}$ where $i$ and $j$ are the two first particles that collide.
	
	We denote $k$ the last creation such that the particles $i$ or $j$ are deviated (it depends only on collision parameters $(a,\gr{\sigma})$). 
	
	\paragraph{We consider first the case where only one particle (say $i$) is  involved in the creation $k$.}
	To deal with the periodicity of $\Lambda$, we consider the covering $\tilde{\Lambda}:=[0,1]\times\mathbb{R}^2{\overset{p}{\rightarrow}}[0,1]\times\mathbb{T}^2=\Lambda$. Particles $i$, $j$ and $k$ have infinitely many copies $(i_{\k})$, $(j_\k)$, $(k_\k)$, with coordinates $(x_{i_\k}(t),v_{i_k}(t)) := (x_{i}(t)+\k,v_{i}(t))$ for $\k\in\{0\}\times\mathbb{Z}^2$. If some $i_\k$ has a recollision with a particle $j_{\k'}$ in $\tilde{\Lambda}$, then the particle $i$ and $j$ have a recollision in $\Lambda$. Because $v_i$ and $v_j$ are bounded by $E^{1/2}$, the particle $i_{\vec{0}}$ can interact only with the $j_\k$ for $|\k|\leq 2TE^{1/2}$. In the following we denote $x_i:=x_{i_{\vec{0}}}$, $x_k:=x_{k_{\vec{0}}}$ and $x_j:=x_{j_{\k}}$.
	
	First we cut the reflection times $t_k$ such that $|x_i(t_k)-x_j(t_k)|\leq \epsilon^{1/3}$ and $d(x_i(t_k),\partial\tilde{\Lambda})<\epsilon^{1/3}$. We observe that $x_k$ and $x_j$ are polygonal trajectories, with at most $TE^{1/2}$ branches. Thus the two particles can approach each other at most $T^2E$ times at distance less than $\epsilon^{1/3}+\epsilon$. Thanks to the condition on $|v_i(t_k^-)\cdot\e|$ and $|v_k(t_k^-)-v_j(t_k^-)|$, for $t_k$ outside a set of size $CT^2E\epsilon^{1/3-1/4}$, the particle $k$ does not come close to the boundary nor to the particle $j$.
	
	Then we consider the "virtual" particle $\tj$ as a particle which moves along straight lines and coincides with $j$ after its last reflection before the recollision, and we denote $(x_\tj(t),v_\tj):=(x_j(\tau_j^-)+(t-\tau_j)v_j(\tau_j^-),v_j(\tau^-_j))$ its coordinates, for $\tau_j$ the time of the last reflection of $j$ before the recollision. If there is a recollision between $i$ and $j$, there is a recollision between $i$ and $\tj$. Because $d(x_i(t_k),\partial\tilde{\Lambda})$ and $|x_i(t_k)-x_j(t_k)|$ are greater than $\epsilon^{1/3}$,  $|x_\tj(t_k)-x_i(t_k)|>\epsilon^{1/3}$.
	
	If $i$ has no reflection until the recollision, then there exists a time $\tau \in (0,t_k)$ and a direction $\nu_{\mathrm{rec}}\in \mathbb{S}^2$ such that :
	\[(x_\tj(t_k)-x_i(t_k)) + (\tau-t_k)(v_\tj -v_i(t_k^-)) = \epsilon \nu_{\mathrm{rec}}.\]
	Thus $(v_i(t_k^-)-v_\tj)$ is in a cone $C((x_\tj(t_k)-x_i(t_k)),\alpha)$ of axes $(x_\tj(t_k)-x_i(t_k))$ and angle $\alpha:=2\arcsin\left( {\epsilon}/{|x_\tj(t_k)-x_i(t_k)|}\right)$. Because ${|x_\tj(t_k)-x_i(t_k)|}>\epsilon^{1/3}$ and $|v_i(t_k)|<E^{1/2}$, $|v_i(t_k)|$ has to be in a rectangle of size $E^{1/2}\times\left( CE^{1/2}\epsilon^{1-1/3}\right)^2$. Hence the size of the set leading to such recollisions goes to zero.
	
	If $i$ has at least one reflection, consider $\omega_i^l$ its last reflection before the recollision, $\tau_l$ the time of the last reflection and $\tau$ the time of recollision. In order to have a recollision, there exists a direction $\nu_{\mathrm{rec}}$ such that
	\begin{multline*}(x_\tj(\tau_l)-x_i(\tau_l)) + (\tau-\tau_l)(v_\tj -|v_i(t_k^-)|\omega_i^l) = \epsilon \nu_{\mathrm{rec}}\\
	\Rightarrow \omega_i^l = \frac{v_\tj}{|v_i(t_k^-)|}+\frac{1}{|v_i(t_k^-)|(\tau_l-\tau)}\left((x_i(\tau_l)-x_\tj(\tau_l))+\epsilon\nu_{\mathrm{rec}}\right).\end{multline*}
	Thus $\omega^l_i$ is in $\left({v_\tj}/{|v_i(t_k^-)|}+C((x_i(\tau_l)-x_\tj(\tau_l)),2\arcsin {\epsilon}/{|x_i(\tau_l)-x_\tj(\tau_l)|})\right)\cap\mathbb{S}^2$. Because the norm of the velocities lays between $\epsilon^{1/4}$ and $E^{1/2}$, for $d := |x_i(\tau_l)-x_\tj(\tau_l)|$, $\omega^l_i$ stays in the intersection of $\mathbb{S}^2$ and a cylinder of radius
	\[\left(1+\frac{|v_\tj|}{|v_i(t_k^-)|}\right)2\tan \left(\arcsin\frac{\epsilon}{|x_i(\tau_l)-x_\tj(\tau_l)|}\right)=CE^{1/2} \frac{\epsilon^{1-1/4}/d}{\sqrt{1-\epsilon^2/d^2}}.\]
	Finally the set of bad directions of reflection $\omega_i^l$ is of size at most :
	\begin{equation}\label{estimation des directions de recollision}
	CE^{1/4}\frac{\left(\epsilon^{1-1/4}/d\right)^{1/2}}{\left(1-\epsilon^2/d^2\right)^{1/4}}\leq CE^{1/4}\frac{\left(\epsilon^{3/4}/d\right)^{1/2}}{\left(1-\epsilon^2/(T^2E)\right)^{1/4}}\leq CE^{1/2}\left(\epsilon^{3/4}/d\right)^{1/2}
	\end{equation}
	using that a particle can cross at most a distance $TE^{1/2}$ and that $\epsilon$ is small enough. In order to control the size of the bad set, we have to cut the trajectories such that $d<\epsilon^{1/2}$. Let $\{x_\tj^\para\}$ the projection of the straight lines $\{x_\tj(t),~t\in\mathbb{R}\}$ on $\partial\tilde{\Lambda}$.
	The last change of direction of the particle $i$ is at the previous reflection $\omega^{l-1}_i$ or at the creation $k$, at the point $X_i$. Because the distance between $X_i$ and  is greater than $\epsilon^{1/3}$, the particle $i$ reaches the boundary at distance less than $\epsilon^{1/2}$ of $\{x_\tj^\para\}$only  if its velocity $v_i(\tau_l^+)$ forms an angle less than $C\epsilon^{1/2-1/3}$ with the plane passing by $\{x_\tj^\para\}$ and $X_i$. Integrating on the $\omega^{l-1}_i$ or $(\nu_k,\bar{v}_k)$ according to the nature of the previous change of direction, the size of parameters such that $d<\epsilon^{1/2}$ goes to zero.
	
	\paragraph{Now we treat the case where $i$ and $j$ are both involved in the creation $k$ (say $k=i$).} We begin by cutting the time where $x_j(t_i)$ is close to the boundary. Because $|v_j(t_i^+)\cdot\e|>\epsilon^{1/4}$, for $t_i$ outside a set of size $CE^{1/2}T \epsilon^{1/3-1/4}$, $x_j$ is at distance greater than $(\epsilon^{1/3}+\epsilon)$ of the boundary.
	
	In the case where neither $i$ nor $j$ have a reflection, the particles do not see the boundary and we can treat it as in the case where the domain is $\mathbb{R}^3$.
	
	Note that in the precedent paragraph, if $i$ had two reflections, we parametrized the bad set by the two last directions of reflection. Thus the same reasoning works and the size of parameters leading to such recollisions  goes to zero.
	
	We have finally to deal with two cases: if the two particles have a reflection or if only one does.
	
	We treat first the case where only $i$ has a reflection which occurs at time $\tau_l$. Because $x_i(t_i^-)$ is at distance at least $\epsilon^{1/3}$ of $\partial\Lambda$ and $|v_i(t_i^-)\cdot\e|<E^{1/2}$, $(t_i-\tau_l)$ is greater than $\epsilon^{1/3}E^{-1/2}$. The distance between $x_i(\tau_l)$ and $x_j(\tau_l)$ is greater than 
	\[d\geq\left\vert\left(\epsilon\nu_i+(\tau_l-t_i)\left(v_i(t_i^-)-v_j(t_i^-)\right)\right)\cdot\e\right\vert > E^{-1/2} \epsilon^{1/3}\epsilon^{1/4}-\epsilon>\epsilon^{1/2}.\]
	Thus we can use the estimation \eqref{estimation des directions de recollision}. The case where $j$ has a reflection can be treated similarly.
	
	In the case where the two particles have a reflection, we suppose that the last reflection involves the particle $i$. In the covering domain $\tilde{\Lambda}$, we apply the formula \eqref{estimation des directions de recollision} with
	\[d = \left|\vec{d}_0 + (\tau_{\tilde{l}}-\tau_l) |v_j|\omega_j^{\tilde{l}}\right|\]
	where $\tau_l$ (respectively $\tau_{\tilde{l}}$) is the time of the last reflection of the particle $i$ (respectively $j$), $\omega_j^{\tilde{l}}$ is the last direction of reflection, and $\vec{d}_0 := x_j(\tau_{\tilde{l}})-x_i(\tau_{l})$. Denoting $\theta$ the angle between $\vec{d_0}$ and $\omega_j^{\tj}$, we can use the two orthogonal decompositions :
	\[\begin{split}
		\vec{d}_0 + (\tau_{\tilde{l}}-\tau_l) |v_j|\omega_j^{\tilde{l}} &= \left(\vec{d}_0\cdot \omega^{\tilde{l}}_j + (\tau_{\tilde{l}}-\tau_l) |v_j|\right)\omega_j^{\tilde{l}}+\vec{d}_0^\perp \\
		&= \left(|\vec{d}_0| + (\tau_{\tilde{l}}-\tau_l) |v_j|\omega_j^{\tilde{l}}\cdot\frac{\vec{d}_0}{|\vec{d}_0|}\right)\frac{\vec{d}_0}{|\vec{d}_0|} + (\tau_{\tilde{l}}-\tau_l) |v_j|\omega_j^{\tilde{l}\perp}.
	\end{split}\]
	Thus 
	\[	d\geq \sup \left(\left|\vec{d}_0^\perp\right|,\left|(\tau_{\tilde{l}}-\tau_l) |v_j|\omega_j^{\tilde{l}\perp}\right|\right)\geq |\sin\theta| \sup\left(\left|\vec{d}_0\right|,\left|\tau_{\tilde{l}}-\tau_l\right| |v_j|\right),\]
	and the size of bad directions $\omega_i^l$ is at most 
	\[\frac{{CE^{1/4}\left(\epsilon^{1-1/4}\right)^{1/2}}}{|\sin \theta|^{1/2}}\inf\left({|\vec{d}_0|}^{-1/2},{(\left|\tau_{\tilde{l}}-\tau_l\right| |v_j|)^{-1/2}}\right).\]
	The term $|\sin\theta|^{-1/2}$ is integrable with respect to the measure $d\omega^{\tilde{l}}_j$. Hence the set of bad directions $(\omega_i^l,\omega_j^{\tilde{l}})$ is of size at most 
	\[CE^{1/4}\epsilon^{3/8}\inf\left({|\vec{d}_0|}^{-1/2},{\left(\left|\tau_{\tilde{l}}-\tau_l\right| |v_j|\right)^{-1/2}}\right).\]
	If $i$ and $j$ have a reflection on a different component of $\partial\Lambda$, $|\vec{d}_0|\geq 1$. Else, denoting $h$ (respectively $h+\delta h$) the distance of $j$ (respectively $i$) from the side of $\partial\Lambda$ where reflections occur,
	\[\begin{split}
	\left||v_j|(\tau_{\tilde{l}}-\tau_l)\right|&=|v_j| \left|\frac{h}{v_j\cdot\e}-\frac{h+\delta h}{v_i\cdot\e}\right|\\
	&\geq \frac{h|(v_i-v_j)\cdot\e|}{|v_i\cdot\e|}-\frac{|\delta h|}{|v_j\cdot\e|}\\
	&\geq E^{1/2}\epsilon^{1/3+1/4}-\epsilon^{1-1/4}\\
	&\geq \epsilon^{7/12}
	\end{split}\]
	for $E$ large enough.
	Hence the set of the bad parameters is of size $C(E,R,T) \epsilon^{3/8-7/24} = C(E,R,T) \epsilon^{1/12}$.
	
	This allows us to conclude that $|\mathcal{P}_2|$ converges to $0$. One can estimate the set of overlaps $\mathcal{P}_2'$ and $\mathcal{P}_3$ in the same way. This concludes the proof of Theorem \ref{theorem de Landford}.\qed

	\begin{appendix}                                     
		\section{Calerman's parametrization and scattering estimates}\label{coordonné de Calerman}
		In section \ref{Term by term convergence} we need to estimate some singular integrals with respect to the measure $|(v-v_*)\cdot\nu|d\nu dv_*$ where $v_i$ can represent $v_*$, $v'$ or $v_*'$. In this appendix we give the detailed statement and proofs.
		                                                    
		We start by recalling the Calerman's collision parameters
		\begin{equation}
		\left\lbrace\begin{split}
		\mathbb{R}^3\times\mathbb{S}^2&\rightarrow C:=\left\lbrace(v',v'_*)\in\mathbb{R}^3\times\mathbb{R}^3,~(v'-v)\cdot(v_*'-v)=0\right\rbrace\\
		(v_*,\nu)&\mapsto(v',v'_*)
		\end{split}\right.
		\end{equation}
		which map the measure $|(v-v_*)\cdot\nu|dv_*d\nu$ into $dv'dS(v'_*)$ where $(v',v_*')$ are given by the scattering and $dS$ is the Lebesgue measure on the affine plane passing through $v$ and normal to $(v'-v)$.
		
		We can then prove our first estimation lemma:
		\begin{lemma}
			Fix $v\in \mathbb{R}^3$ and $a<b$ two real numbers. Then for $v_i$ equal to $v_*$, $v'$ or $v_*'$, for some contant $C>0$
			\begin{equation}
			\int_{\mathbb{R}^3\times\mathbb{S}^2} \ind_{(v_i\cdot\e)\in[a,b]}\ind_{|v|^2+|v_*|^2\leq E} |(v-v_*)\cdot\nu|dv_*d\nu \leq C E^2 |b-a|.
			\end{equation}
		\end{lemma}
		\begin{proof}
			In the case where $v_i = v_*$, the proof is straightforward. If $v_i=v'$ we apply Calerman's change of variables and the proof is also direct.
			
			In the case where $v_i=v_*'$, we apply Calerman's change of variables. If $v_*'\cdot\e$ is between $a$ and $b$, then $v'_*$ stays in a rectangle of size $E^{1/2}\times|b-a|/|\sin(\e,v'-v)|$. Then integrating with respect to $v'$, we obtain the expected bound.
		\end{proof}
		From this lemma we can deduce the following estimates:
		\begin{prop}
			Fix $v\in \mathbb{R}^3$. Then for $v_i$ equal to $v_*$, $v'$ or $v_*'$, for some contant $C>0$
			\begin{equation}\label{intégration de 1/v}
			\int_{\mathbb{R}^3\times\mathbb{S}^2} \left(\frac{\epsilon}{|v_i\cdot\e|}\wedge 1\right)\ind_{|v|^2+|v_*|^2\leq E} |(v-v_*)\cdot\nu|dv_*d\nu \leq C E^2 \epsilon|\log \epsilon|
			\end{equation}
			\begin{equation}\label{intégration de 1/v^2}
			\int_{\mathbb{R}^3\times\mathbb{S}^2} \left(\frac{\epsilon}{|v_i\cdot\e|^2}\wedge 1\right)\ind_{|v|^2+|v_*|^2\leq E} |(v-v_*)\cdot\nu|dv_*d\nu \leq C E^2 \epsilon^{1/2}
			\end{equation}
		\end{prop}
		\begin{proof}
			We use the previous lemma to decompose the set:
			\[\begin{split}
			\int_{\mathbb{R}^3\times\mathbb{S}^2}& \left(\frac{\epsilon}{|v_i\cdot\e|}\wedge 1\right)\ind_{|v|^2+|v_*|^2\leq E} |(v-v_*)\cdot\nu|dv_*d\nu \\
			&\leq \int_{\mathbb{R}^3\times\mathbb{S}^2} \ind_{|v_i\cdot\e|\in[0,\epsilon]}\ind_{|v|^2+|v_*|^2\leq E} |(v-v_*)\cdot\nu|dv_*d\nu\\
			&~~+\sum_{1\leq n\leq E^{1/2}/\epsilon}\frac{1}{n}\int_{\mathbb{R}^3\times\mathbb{S}^2} \ind_{|v_i\cdot\e|\in[n\epsilon,(n+1)\epsilon]}\ind_{|v|^2+|v_*|^2\leq E} |(v-v_*)\cdot\nu|dv_*d\nu\\
			&\leq CE^2\epsilon\left( 1 + \sum_{1\leq n\leq E^{1/2}/\epsilon}\frac{1}{n}\right)\\
			&\leq CE^2\epsilon |\log E^{1/2}/\epsilon|\leq CE^2\epsilon |\log \epsilon|
			\end{split}\]
			
			The proof of the second line is similar.
			\[\begin{split}
			\int_{\mathbb{R}^3\times\mathbb{S}^2}& \left(\frac{\epsilon}{|v_i\cdot\e|^2}\wedge 1\right)\ind_{|v|^2+|v_*|^2\leq E} |(v-v_*)\cdot\nu|dv_*d\nu \\
			&\leq \int_{\mathbb{R}^3\times\mathbb{S}^2} \ind_{|v_i\cdot\e|\in[0,\epsilon^{1/2}]}\ind_{|v|^2+|v_*|^2\leq E} |(v-v_*)\cdot\nu|dv_*d\nu\\
			&~~+\sum_{1\leq n} \frac{1}{n^2}\int_{\mathbb{R}^3\times\mathbb{S}^2} \ind_{|v_i\cdot\e|\in[n\epsilon^1/2,(n+1)\epsilon^{1/2}]}\ind_{|v|^2+|v_*|^2\leq E} |(v-v_*)\cdot\nu|dv_*d\nu\\
			&\leq CE^2\epsilon^{1/2}\left( 1 + \sum_{1\leq n}\frac{1}{n^2}\right)\\
			&\leq CE^2\epsilon^{1/2}
			\end{split}\]
		\end{proof}
	
	\end{appendix}                                       
	\bibliographystyle{abbrv}
	\bibliography{reference_rapport}
\end{document}